\documentclass[11pt]{article}

\usepackage{amsmath,amssymb,amsthm,latexsym}

\usepackage{color}
\usepackage{comment}

\allowdisplaybreaks

\newtheorem{theorem}{Theorem}[section]

\newtheorem{corollary}[theorem]{Corollary}

\newtheorem{lemma}[theorem]{Lemma}
\newtheorem{proposition}[theorem]{Proposition}

\newtheorem{definition}[theorem]{Definition}
\newtheorem{remark}[theorem]{Remark}

\numberwithin{equation}{section}
\newtheorem{assumption}[theorem]{Assumption}

\def\be{\begin{equation}}
\def\ee{\end{equation}}
\def\bes{\begin{equation*}}
\def\ees{\end{equation*}}

\def\vp{{\varphi}}

\def\wt{\widetilde}

\def\eps{\varepsilon}
\def\lam{{\lambda}}
\def\ol{\overline}

\def\qed{{\hfill $\square$ \bigskip}}

\def\esssup{{\mathop{\rm ess \; sup \, }}}
\def\essinf{{\mathop{\rm ess \; inf \, }}}
\def\essosc{{\mathop{\rm ess \; osc \, }}}

\def\sE {{\cal E}} \def\sF {{\cal F}}\def\sL {{\cal L}}
\def\sN {{\cal N}}

\def\bE {{\mathbb E}} \def\bN {{\mathbb N}}
\def\bP {{\mathbb P}} \def\bR {{\mathbb R}}

\def\EHR{\mathrm{EHR}}
\def\PHR{\mathrm{PHR}}
\def\PHI{\mathrm{PHI}}

\def\Ca{\mathrm{Cap}}
\def\Ex{\mathrm{Ext}}

\def\CSJ{\mathrm{CSJ}}

\def\PI{\mathrm{PI}}

\def\IJ{\mathrm{IJ}}
\def\FK{\mathrm{FK}}
\def\VD{\mathrm{VD}}
\def\RVD{\mathrm{RVD}}

\def\PHI{\mathrm{PHI}}

\definecolor{dred}{rgb}{0.8, 0.0, 0.0}

\def\UJS{\mathrm{UJS}}

\def\EHI{\mathrm{EHI}}

\def\WEHI{\mathrm{WEHI}}
\def\E{\mathrm{E}}

\def\J{\mathrm{J}}
\def\<{\langle}
\def\>{\rangle}
\def\FF{{\cal F}}

\def\T{\mathrm{Tail}}

\begin{document}

\title{\bf Elliptic Harnack inequalities for symmetric non-local Dirichlet forms}

\author{Zhen-Qing Chen,
\quad Takashi Kumagai\footnote{Research
partially supported by the Grant-in-Aid for Scientific Research (A)
25247007 and 17H01093, Japan.}
\quad{and} \quad
 Jian Wang\footnote{Research partially supported by the National
Natural Science Foundation of China (No.\ 11522106), the JSPS postdoctoral fellowship
(26$\cdot$04021), the Fok Ying Tung
Education Foundation (No.\ 151002), National Science Foundation of
Fujian Province (No.\ 2015J01003), the Program for Probability and Statistics: Theory and Application (No.\ IRTL1704), and Fujian Provincial
Key Laboratory of Mathematical Analysis and its Applications
(FJKLMAA).}}

\date{}
\maketitle

\begin{abstract}
 We study relations and  characterizations of various elliptic Harnack inequalities for symmetric non-local
 Dirichlet forms on metric measure spaces.
 We allow the scaling function be state-dependent and  the state space possibly disconnected.
 Stability of elliptic Harnack inequalities is established under certain  regularity conditions
 and implication for a priori H\"older regularity of harmonic functions is explored.
 New equivalent statements
 for parabolic Harnack inequalities of non-local Dirichlet forms are obtained in terms of elliptic Harnack inequalities.
\end{abstract}

\section{Introduction and Main Results} \label{sec:intro}

The classical elliptic Harnack inequality
asserts that there exists a universal constant $c_1=c_1(d)$ such that for every $x\in \bR^d$, $r>0$
and every non-negative harmonic function $h$ in the ball $B(x_0, 2r)\subset \bR^d$,
\begin{equation}\label{e:1.1}
\esssup_{ B(x_0,r)} h\le c\, \essinf_{ B(x_0,r)} h.
\end{equation}
A celebrated theorem of Moser (\cite{Mos})  says that
such elliptic Harnack inequality holds for non-negative harmonic functions of any
uniformly elliptic divergence operator on
$\bR^d$. One of the important consequences of Moser's
 elliptic Harnack inequality   is
that it implies a priori elliptic H\"older regularity
(see Definition \ref{thm:defHR} below) for harmonic functions
of uniformly elliptic operators of divergence form.
 Because of the fundamental importance role played by a priori elliptic H\"older regularity  for
 solutions of elliptic and parabolic differential equations, elliptic Harnack inequality  and  parabolic Harnack inequality, which is a parabolic version of the Harnack inequality (see Definition \ref{thm:defPHI} below), have been investigated extensively for local operators (diffusions) on various spaces such as
manifolds, graphs and metric measure spaces. It is also very important to consider whether such Harnack inequalities are stable under perturbations of
the associated quadratic forms and under rough isometries. The stability problem of elliptic Harnack inequality  is a difficult one. In \cite{B}, R. Bass proved  stability
of elliptic Harnack inequality   under some strong global bounded geometry condition.
Quite recently, this
assumption has been relaxed significantly by Barlow-Murugan (\cite{BM}) to bounded geometry condition.

 For non-local operators, or equivalently, for discontinuous Markov processes,
harmonic functions are required to be non-negative on the whole space in the formulation of  Harnack inequalities
due to the jumps from the
processes; see $\EHI$ (elliptic Harnack inequality) in Definition \ref{thmdefEHIq}(i) below.
In Bass and Levin \cite{BL}, this version of $\EHI$ has been established for a class of non-local operators.
If we only require harmonic functions to be non-negative
in the ball $B(x_0, 2r)$, then the classical elliptic Harnack inequality  \eqref{e:1.1} does not need to hold. Indeed, Kassmann \cite{K1} constructed such a
counterexample for fractional Laplacian on $\bR^d$.
On the other hand,
for non-local operators,
   it is not known whether $\EHI$ implies a priori
    elliptic H\"older regularity
($\EHR$) and we suspect it is not (although parabolic Harnack inequality  ($\PHI$) does imply
  parabolic H\"older regularity and hence $\EHR$, see Theorem \ref{T:PHI} below).
To address this problem, some versions
 of elliptic Harnack inequalities that  imply $\EHR$ are considered in some literatures such as \cite{CKP,CKP2,K1},
in connection with the
Moser's iteration method.  We note that there are now many
 related work on $\EHI$ and $\EHR$ for harmonic functions of non-local operators; in addition to the papers mentioned above; see, for instance, \cite{BS,CK1,CK2,CS,DK,GHH,Ha,HN,K2,Sil} and references therein.
This is only a partial list of the vast literature on the subject.

The aim of this paper is to investigate  relations among various elliptic Harnack inequalities   and to study their stability for symmetric non-local Dirichlet forms
 under a general setting of metric measure spaces.
This paper can be regarded as a continuation of \cite{CKW1,CKW2}, which are concerned with the stability of two-sided heat kernel estimates,
upper bound heat kernel estimates and   $\PHI$ for non-local Dirichlet forms on general metric measure spaces.
We point out that the setting of this paper is much more general than that of \cite{CKW1, CKW2} in the sense that the scale function
in this paper can
 be state-dependent. As a byproduct, we obtain new equivalent statements for $\PHI$ in terms of $\EHI$.

\subsection{Elliptic Harnack inequalities}

Let $(M,d)$ be a locally compact separable metric space, and
$\mu$ a positive Radon measure on $M$ with full support.
We will refer to such a triple $(M, d, \mu)$ as a \emph{metric measure space}.
Throughout the paper, we assume that $\mu (M)=\infty$.
We   emphasize that we do not assume $M$ to be connected nor
 $(M, d)$ to be geodesic.

We consider a regular symmetric \emph{Dirichlet form} $(\sE, \sF)$
on $L^2(M; \mu)$ of pure jump type; that is,
$$
\sE(f,g)=\int_{M\times M\setminus \textrm{diag}}(f(x)-f(y))(g(x)-g(y))\,J(dx,dy),\quad f,g\in \sF,
$$
where \textrm{diag} denotes the diagonal set $\{(x, x): x\in M\}$
and $J(\cdot,\cdot)$ is a symmetric Radon measure   on $M\times M\setminus \textrm{diag}$.
Since $(\sE,\sF)$ is regular, each function $f\in \sF$ admits a
quasi-continuous version $\tilde f$ (see
\cite[Theorem 2.1.3]{FOT}). In the paper, we will always represent
$f\in \FF$ by its quasi-continuous version  without writing $\tilde f$.
Let
$\sL$ be the (negative definite) $L^2$-\emph{generator} of $\sE$; this is,
$\sL$  is the
self-adjoint operator in $L^2(M; \mu)$ such that
$$
 \sE (f,g) = -\langle \sL f,g \rangle \quad \hbox{for all }
f \in \mathcal{D}(\sL) \hbox{ and } g \in \sF,
$$ where
$\langle  \cdot , \cdot \rangle$ denotes the inner product in $L^2(M; \mu)$.
  Let $\{P_t\}_{t\geq 0}$  be
its associated $L^2$-\emph{semigroup}.
 Associated with the regular
Dirichlet form $(\sE, \sF)$ on $L^2(M; \mu)$ is an $\mu$-symmetric
\emph{Hunt process} $X=\{X_t, t \ge 0;  \bP^x, x \in M\setminus
\sN\}$. Here $\sN$ is a properly exceptional set for $(\sE, \sF)$ in
the sense that
 $\mu(\sN)=0$ and
$\bP^x( X_t \in \sN \hbox{ for some } t>0 )=0$ for all $x \in M\setminus\sN$. This Hunt process is unique up to a
properly exceptional set --- see \cite[Theorem 4.2.8]{FOT}.
We fix $X$ and $\sN$, and write $ M_0 = M \setminus \sN.$
While the semigroup $\{P_t,   t\ge0\} $ associated with $\sE$ is defined on $L^2(M; \mu)$,
a more precise version with better regularity properties can be obtained,
if we set, for any bounded Borel measurable function $f$ on $M$,
$$ P_t f(x) = \bE^x f(X_t), \quad x \in M_0.
$$

\begin{definition}\label{thm:defvdrvd} \rm Denote by $B(x,r)$ the ball in $(M,d)$ centered at $x$ with radius $r$, and set
\[
V(x,r) = \mu(B(x,r)).
\]
(i) We say that $(M,d,\mu)$ satisfies the {\it volume doubling property} ($\VD$) if
there exists a constant ${C_\mu}\ge 1$ such that
for all $x \in M$ and $r > 0$,
\be \label{e:vd}
 V(x,2r) \le {C_\mu} V(x,r).
 \ee
(ii) We say that $(M,d,\mu)$ satisfies the {\it reverse volume doubling property} ($\RVD$)
if there exist positive constants
$d_1$ and ${c_\mu}$ such that
for all $x \in M$ and $0<r\le R$,
\be \label{e:rvd}
\frac{V(x,R)}{V(x,r)}\ge {c_\mu}  \Big(\frac Rr\Big)^{d_1}.
\ee
 \end{definition}

Since $\mu$ has full support on $M$, we have $V(x,r)=\mu (B(x, r))>0$
for every $x\in M$ and $r>0$.
The  VD condition \eqref{e:vd} is equivalent to the existence of $d_2>0$ and
${\wt C_\mu}\ge1$ so that
\begin{equation}\label{e:vd2}
\frac{V(x,R)}{V(x,r)} \leq   {\wt C_\mu}  \Big(\frac Rr\Big)^{d_2} \quad \hbox{for all } x\in M \hbox{ and }
0<r\le R.
\end{equation}
 The RVD condition \eqref{e:rvd} is equivalent to the existence of constants $l_\mu >1$ and $\wt c_\mu >1$ so that
 $$
 V (x, l_\mu r)  \geq \wt c_\mu V (x, r) \quad \hbox{for every } x \in M \hbox{ and } r>0.
 $$

It is  known that $\VD$ implies
$\RVD$
if $M$ is connected and unbounded.
In fact, it also holds that if $M$ is connected
and \eqref{e:vd} holds for
all $x\in M$ and $r\in (0,R_0]$
with some $R_0>0$, then
\eqref{e:rvd} holds for
all $x\in M$ and $0<r\le R\le R_0$.
See, for example \cite[Proposition 5.1 and Corollary 5.3]{GH1}.

\medskip

Let $\bR_+:=[0,\infty)$ and $\phi: M\times \bR_+\to \bR_+$ be  a strictly increasing continuous function for every fixed $x\in M$
with $\phi (x, 0)=0$ and $\phi(x, 1)=1$ for all $x\in M$, and satisfying that
\begin{itemize}
\item[(i)] there exist constants $c_1,c_2>0$ and $\beta_2\ge \beta_1>0$ such that
\be\label{pvd1}
 c_1 \Big(\frac Rr\Big)^{\beta_1} \leq
\frac{\phi (x,R)}{\phi (x,r)}  \ \leq \ c_2 \Big(\frac
Rr\Big)^{\beta_2}
\quad \hbox{for all } x\in M\mbox{ and }
0<r \le R;
\ee

\item[(ii)]
there exists a constant $c_3\ge1$ such that
\begin{equation}\label{pvd}\phi(y,r)\le c_3\phi(x,r)\quad \hbox{for all } x,y \in M\mbox{ with }d(x,y)\le r.\end{equation}
\end{itemize}

Recall that
a set $A\subset M$ is said to be nearly Borel measurable
if for any probability measure $\mu_0$ on $M$, there are
Borel measurable subsets $A_1$, $A_2$ of $M$ so that
$A_1\subset A\subset A_2$ and that $\bP^{\mu_0} (X_t\in A_2\setminus A_1
\hbox{ for some } t\geq 0)=0$. The collection of all nearly Borel measurable subsets
of $M$ forms a $\sigma$-field, which is called nearly Borel measurable
$\sigma$-field. A nearly Borel measurable function $u$ on $M$ is said to be \emph{subharmonic}  (resp. \emph{harmonic, superharmonic})
in ${D}$ (with respect to the process $X$)
if for any relatively compact subset $U\subset D$,
$t\mapsto   u (X_{t\wedge \tau_U}) $ is a uniformly integrable submartingale
(resp. martingale, supermartingale) under $\bP^x$ for
$\sE$-q.e. $x\in U$.  Here $\sE$-q.e. stands for $\sE$-quasi-everywhere, meaning it holds outside a set having
zero 1-capacity with respect to the Dirichlet form $(\sE, \sF)$;  see \cite{CF, FOT} for its definition.

For a Borel measurable function $u$ on $M$, we define its
\emph{non-local tail}
$ {\T_\phi} (u; x_0,r)$
 in the ball
$B(x_0,r)$ by
$$
 {\T_\phi} \, (u; x_0,r):=\int_{B(x_0,r)^c}\frac{|u(z)|}{V(x_0,d(x_0,z))\phi(x_0,d(x_0,z))}\,\mu(dz).
$$ We need the following definitions for various forms of elliptic Harnack inequalities.

\begin{definition}\label{thmdefEHIq}
\rm \begin{description}
\item{(i)} We say that \emph{elliptic Harnack inequality} ($\EHI$) holds for the process $X$, if there exist constants $\delta\in(0,1)$ and  $c\ge1$ such that   for every $x_0 \in M $, $r>0$ and for
 every non-negative measurable function $u$ on $M$ that is harmonic in $B(x_0,r)$,
$$
\esssup_{ B(x_0,\delta r)}
u\le c\, \essinf_{ B(x_0,\delta r)}
u.
$$

\item{(ii)} We say that \emph{non-local elliptic Harnack inequality}
$\EHI(\phi)$
holds if there exist constants $\delta\in
(0,1)$ and $c\ge1$ such that for every $x_0 \in M $, $R>0$,  $0<r\le \delta R$, and any measurable function $u$ on $M$ that is non-negative and  harmonic
in $B(x_0,R)$,
$$\esssup_{B(x_0,r)}u\le c\Big(\essinf_{B(x_0,r)}u+\phi(x_0,r){\T_\phi}\, (u_-; x_0,R)\Big).$$

\item{(iii)} We say that \emph{non-local weak elliptic Harnack inequality} $\WEHI(\phi)$ holds if there exist constants $\varepsilon$, $\delta\in
(0,1)$ and $c\ge1$ such that for every $x_0 \in M $, $R>0$,  $0<r\le \delta R$, and any measurable function $u$ on $M$ that is non-negative and  harmonic in $B(x_0,R)$,
$$\left(\frac{1}{\mu(B(x_0,r))}\int_{B(x_0,r)} u^\varepsilon\,d\mu\right)^{1/\varepsilon}\le c\Big(\essinf_{B(x_0,r)}u+\phi(x_0,r){\T_\phi}\, (u_-; x_0,R)\Big).$$

\item{(iv)} We say that \emph{non-local weak elliptic Harnack inequality}
$\WEHI^+(\phi)$
 holds if (iii) holds for any measurable function $u$ on $M$ that is non-negative and  superharmonic in $B(x_0,R)$.
\end{description}
\end{definition}

Clearly, $\EHI(\phi)\Longrightarrow \EHI+ \WEHI(\phi)$, and $\WEHI^+(\phi)\Longrightarrow \WEHI(\phi)$.
We note that unlike the diffusion case, one needs
to assume in the definition of $\EHI$  that the harmonic function $u$ is non-negative on the whole space $M$ because the process
$X$ can jump all over the places,
 as mentioned at the beginning of this  section.

\begin{remark} \rm \begin{itemize}
\item[(i)]
For strongly local Dirichlet forms, $\EHI(\phi)$ is just $\EHI$, and $\WEHI^+(\phi)$ (resp. $\WEHI(\phi)$) is simply reduced into the following:
there exist constants $\varepsilon,\delta\in(0,1)$ and $c\ge1$ such that  for every $x_0 \in M $, $0<r\le \delta R$, and for
 every measurable function $u$ that is non-negative and  superharmonic (resp.\, harmonic) in $B(x_0,R)$,  $$
\left(\frac{1}{\mu(B(x_0,r))}\int_{B(x_0,r)} u^\varepsilon \,d\mu\right)^{1/\varepsilon}\le  c \, \essinf_{B(x_0,r)}u.
$$
 The above inequality is called
weak Harnack inequality for differential operators.
This is why $\WEHI(\phi)$ is called
weak Harnack inequality in
\cite{CKP,CKP2,K1}. However for non-local operators this terminology is a bit misleading as it is not implied by $\EHI$.

\item[(ii)]Non-local (weak) elliptic Harnack inequalities have a term involving the non-local tail of harmonic functions, which are essentially due to the
jumps of the symmetric Markov  processes. This new formulation of Harnack inequalities without requiring the additional positivity on the whole space but adding a non-local tail term first appeared in \cite{K1}. The notion of  non-local tail of measurable function is formally  introduced in \cite{CKP,CKP2}, where non-local (weak) elliptic Harnack inequalities and local behaviors of fractional $p$-Laplacians are investigated.
See \cite{DK} and references therein for the background of $\EHI$ and $\WEHI$.
\end{itemize}
\end{remark}

To state relations among various notions of elliptic Harnack inequalities
and their characterizations, we need a few definitions.

\begin{definition}\rm

\begin{itemize}

\item[{(i)}] We say $\J_{\phi}$ holds if there exists a non-negative symmetric
function $J(x, y)$ so that for $\mu\times \mu $-almost  all $x, y \in M$,
\begin{equation}\label{ae:1.2}
J(dx,dy)=J(x, y)\,\mu(dx)\, \mu (dy),
\end{equation} and
\begin{equation}\label{ajsigm}
 \frac{c_1}{V(x,d(x, y)) \phi(x,d(x, y))}\le J(x, y) \le \frac{c_2}{V(x,d(x, y)) \phi(x,d(x, y))}.
 \end{equation}
We say that $\J_{\phi,\le}$ (resp. $\J_{\phi,\ge}$) holds if \eqref{ae:1.2} holds and the upper bound (resp. lower bound) in \eqref{ajsigm} holds for $J(x, y)$.

\item[{(ii)}]We say that $\IJ_{\phi,\le}$ holds if for $\mu$-almost  all $x \in M$ and any $r>0$,$$
 J(x,B(x,r)^c)\le \frac{c_3}{\phi(x,r)}.$$
 \end{itemize}
\end{definition}

\begin{remark}\rm \label{intelem}Under $\VD$ and \eqref{pvd1},  $\J_{\phi,\le}$ implies $\IJ_{\phi,\le}$, see, e.g., \cite[Lemma 2.1]{CKW1} or \cite[Lemma 2.1]{CK2} for a proof.
 \end{remark}

For the non-local Dirichlet form $(\sE, \sF)$, we
define the carr\'e du-Champ operator $\Gamma (f, g)$
for $f, g\in \sF$ by
$$
\Gamma (f, g) (dx) = \int_{y\in M} (f(x)-f(y))(g(x)-g(y))\,J(d x,d y)  .
$$
Clearly $\sE(f,g)=  \Gamma(f,g) (M)$.
For any $f\in \sF_b:=\sF\cap L^\infty(M,\mu)$, $\Gamma(f,f)$ is the unique
Borel measure (called the \emph{energy measure}) on $M$ satisfying
$$
\int_M g \, d\Gamma(f,f)=\sE(f, fg)-\frac 12\sE(f^2,g),\quad f,g\in \sF_b.
$$
Let $U \subset V$ be open sets in $M$ with
$U \subset \ol U \subset V$.
We say a non-negative bounded measurable function $\vp$ is a {\it cutoff function for $U \subset V$},
if $\vp  = 1$ on $U$,  $\vp=0$ on $V^c$ and $0\leq \vp \leq 1$ on $M$.

\medskip

\begin{definition}\label{csj1} \rm
We say that {\it cutoff Sobolev inequality} $\CSJ(\phi)$ holds if there exist constants $C_0\in (0,1]$ and $C_1, C_2>0$
such that for every
$0<r\le R$, almost all
$x_0 \in M$ and any $f\in \sF$, there exists
a cutoff function $\vp\in \sF_b$ for $B(x_0,R) \subset B(x_0,R+r)$ so that
\begin{equation*} \begin{split}
 \int_{B(x_0,R+(1+C_0)r)} f^2 \, d\Gamma (\vp,\vp)
\le &C_1 \int_{U\times U^*}(f(x)-f(y))^2\,J(dx,dy) \\
&+ \frac{C_2}{\phi(x_0,r)}  \int_{B(x_0,R+(1+C_0)r)} f^2  \, d\mu,
\end{split}
\end{equation*}
where $U=B(x_0,R+r)\setminus B(x_0,R)$ and $U^*=B(x_0,R+(1+C_0)r)\setminus B(x_0,R-C_0r)$.

\end{definition}
$\CSJ(\phi)$ is introduced in \cite{CKW1}, and is used to control
the energy of cutoff functions and to characterize the stability of heat kernel estimates for non-local Dirichlet forms. See \cite[Remark 1.6]{CKW1} for  background on
$\CSJ(\phi)$.

\begin{definition}\label{pi1} {\rm
We say that {\em Poincar\'e inequality}
$\PI(\phi)$
holds if there exist constants $C>0 $ and $\kappa\ge1$ such that
for any  ball $B_r=B(x_0,r)$ with $x_0\in M$ and $r>0$, and for any $f \in \sF_b$,
$$
\int_{B_r} (f-\bar{f}_{B_r})^2\, d\mu \le C \phi(x_0,r)\int_{{B_{\kappa r}}\times {B_{\kappa r}}} (f(y)-f(x))^2\,J(dx,dy),
$$
where $\bar{f}_{B_r}= \frac{1}{\mu({B_r})}\int_{B_r} f\,d\mu$ is the average value of $f$ on ${B_r}$.} \end{definition}

We next introduce the modified Faber-Krahn inequality.
For any open set $D \subset M$, let $\sF_D$
be the $\sE_1$-closure in $\sF$ of  $\sF\cap C_c(D)$,
where $\sE_1 (u, u):= \sE (u, u)+\int_M u^2 \,d\mu $.
Define
$$
 \lam_1(D)
= \inf \left\{ \sE(f,f):  \,  f \in \sF_D \hbox{ with }  \|f\|_2 =1 \right\},
$$
the bottom of the Dirichlet spectrum of $-\sL$ on $D$.

\begin{definition}\label{fk1}
{\rm We say that {\em Faber-Krahn
inequality} $\FK(\phi)$ holds, if there exist positive constants $C$ and
$\nu$ such that for any ball $B(x,r)$ and any open set $D \subset
B(x,r)$,$$
 \lam_1 (D) \ge \frac{C}{\phi(x,r)} (V(x,r)/\mu(D))^{\nu}.
$$
} \end{definition}

For a set
$A\subset M$, define the exit time $\tau_A = \inf\{ t >0 : X_t \in A^c \}$.

\begin{definition}{\rm We say that $\E_\phi$ holds if
there is a constant $c_1>1$ such that for all $r>0$ and  all $x\in M_0$,
$$
c_1^{-1}\phi(x,r)\le \bE^x [ \tau_{B(x,r)} ] \le c_1\phi(x,r).
$$
We say that $\E_{\phi,\le}$ (resp. $\E_{\phi,\ge}$) holds
if the upper bound (resp. lower bound) in the inequality above holds. }
\end{definition}

\begin{definition}\label{thm:defHR}\rm We say
 \emph{elliptic H\"older regularity}
($\EHR$)
holds for the process $X$, if there exist constants $c>0$,
$\theta\in (0, 1]$ and $\eps \in (0, 1)$ such that  for every $x_0 \in M $, $r>0$ and for
 every bounded measurable function $u$ on $M$ that is harmonic in $B(x_0, r)$, there is a properly exceptional set ${\cal N}_u\supset {\cal N}$
so that
$$
|u(x)-u(y)| \leq c   \left( \frac{d(x, y)}{r}\right)^\theta\|u\|_\infty $$
for any $x,$ $y \in B(x_0, \eps r)\setminus {\cal N}_u.$\end{definition}

Here is the main result of this paper.

\begin{theorem} \label{T:ehi-1} Assume that the metric measure space $(M, d , \mu)$ satisfies $\VD$, and $\phi$ satisfies \eqref{pvd1} and \eqref{pvd}.
Then  we have
\begin{itemize}

\item[{\rm (i)}] $\WEHI(\phi) \Longrightarrow \EHR;$ \\
$\WEHI(\phi)+\J_{\phi}+\FK(\phi)+\CSJ(\phi)\Longrightarrow \EHI(\phi)$.

\item[{\rm (ii)}] $\J_{\phi,\le}+\FK(\phi)+\PI(\phi)+\CSJ(\phi)\Longrightarrow \WEHI^+(\phi)$.

\item[{\rm (iii)}]
 $\EHI+\E_{\phi,\le}+\J_{\phi,\le}\Longrightarrow \EHI(\phi)+ \FK(\phi);$\\
  $\EHI+\E_{\phi}+\J_{\phi,\le}\Longrightarrow \PI(\phi).$
\end{itemize}
\end{theorem}

As a direct consequence of Theorem \ref{T:ehi-1}, we have the following statement.

\begin{corollary} \label{C:cor-1} Assume that the metric measure space $(M, d , \mu)$ satisfies $\VD$, and $\phi$ satisfies \eqref{pvd1} and \eqref{pvd}.
If $\J_{\phi}$ and $\E_\phi$ hold, then
$$\FK(\phi)+\PI(\phi)+\CSJ(\phi)\Longleftrightarrow \WEHI^+(\phi)\Longleftrightarrow \WEHI(\phi)\Longleftrightarrow\EHI(\phi)\Longleftrightarrow\EHI.$$
\end{corollary}

\begin{proof}
It follows from
Theorem \ref{T:ehi-1}(ii) that, if $\J_{\phi,\le}$ holds, then
$$\FK(\phi)+\PI(\phi)+\CSJ(\phi)\Longrightarrow \WEHI^+(\phi)\Longrightarrow \WEHI(\phi).$$ By Proposition \ref{P:fk}(i) below,
$\EHR+\E_{\phi,\le}\Longrightarrow \FK(\phi)$. On the other hand, according to Proposition \ref{P:CSJ} below, we have $\E_\phi+\J_{\phi, \leq}
\Longrightarrow\CSJ(\phi).$ Combining those with Theorem \ref{T:ehi-1}(i), we can obtain that under $\J_{\phi}$ and $\E_\phi$,
$$ \WEHI(\phi)\Longrightarrow\EHI(\phi)\Longrightarrow\EHI.$$
Furthermore,  by
Theorem \ref{T:ehi-1}(iii), if $\J_{\phi, \leq}$
and $\E_\phi$ are satisfied, then $\EHI\Longrightarrow\FK(\phi)+\PI(\phi).$
As mentioned above, Proposition \ref{P:CSJ} below shows that $\E_\phi+\J_{\phi, \leq}
\Longrightarrow\CSJ(\phi).$ Thus, if $\J_{\phi,\le}$ and $\E_\phi$ are satisfied, then
$$\EHI\Longrightarrow \FK(\phi)+\PI(\phi)+\CSJ(\phi).$$ The proof is complete. \qed\end{proof}

\subsection{Stability of elliptic Harnack inequalities} \label{S:1.2}

In this subsection, we study the stability of $\EHI$ under some additional assumptions. We mainly follow the framework of \cite{B}.
For open subsets $A$ and $B$ of $M$ with
$A\Subset B$ (that is, $A\subset \overline A\subset B$),
define the relative capacity
\[
 \mbox{{\rm Cap}} (A,B)=\inf \left\{\sE(u,u): u\in \sF, \, u=1
 \hbox{ $\sE$-q.e. on } A
 \hbox{ and } u=0
  \hbox{ $\sE$-q.e. on } B^c \right\}.
\]
For each $x\in M$ and $r>0$, define
$$
\Ex(x,r)=V(x,r)/\Ca(B(x,r),B(x,2r)).
$$

Our main assumptions are as follows.

\begin{assumption}\label{A1}
\begin{itemize}
\item [\rm (i)] $(M,\mu)$ satisfies $\VD$ and $\RVD$.

\item [\rm (ii)]
There is a constant $c_1>0$ such that for all $x, y\in M$ with $d(x,y)\le r$,
\[
\Ca(B(y,r),B(y,2r))\le  c_{ 1}\Ca(B(x,r), B(x,2r)).
 \]

\item[\rm (iii)]
For any $a\in (0,1]$, there exists a constant $c_2:=c_{2,a}>0$ such that for all $x\in M$ and $r>0$,
\[
\Ca(B(x,r),B(x,2r))\le c_2\,\Ca(B(x,ar), B(x,2r)).
\]

\item[\rm (iv)] There exist constants $c_3,c_4>0$ and $\beta_2\ge \beta_1>0$ such that for all $x\in M$ and $0<r\le R$,
\begin{equation}\label{vext}
 c_3 \Big(\frac Rr\Big)^{\beta_1} \leq
\frac{\Ex(x,R)}{\Ex(x,r)}  \ \leq \ c_4 \Big(\frac
Rr\Big)^{\beta_2}.
\end{equation}
\end{itemize}
\end{assumption}

 \medskip

\begin{assumption}\label{A1-b}
For any bounded, non-empty open set $D\subset M$, there exist a properly exceptional set $\sN_D\supset \sN$ and  a non-negative
measurable
 function $G_D(x,y)$  defined for $x,y\in D\setminus \sN_D$ such that
 \begin{description}
 \item{\rm (i)} $G_D(x,y)=G_D(y,x)$ for all $(x,y)\in (D\setminus \sN_D) \times (D\setminus \sN_D)
\setminus {\rm diag}$.

\item{\rm (ii)}  for every fixed $y\in D\setminus \sN_D$, the function $x\mapsto G_D(x,y)$ is harmonic in
$ (D\setminus \sN_D)\setminus\{y\} $.

\item{\rm (iii)} for every measurable $f\geq 0$ on $D$,
\[
\bE^x \left[ \int_0^{\tau_D}  f(X_t) \,dt \right]=
\int_DG_D(x,y)f(y)\,\mu(dy),
\quad x\in D\setminus \sN_D.
\]
\end{description}

\end{assumption}

The  function $G_D(x, y)$ satisfying (i)--(iii) of Assumption \ref{A1-b}  is called the Green function of $X$ in $D$.

\begin{remark}\label{R:1.3}
\rm \begin{description} \item {\rm (i)}
We will see from Lemmas \ref{L:ug} and \ref{L:lg} below that, under suitable conditions,
the quantity $\Ex (x, r)$ defined above is related to the mean exit time
from the ball $B(x, r)$ by the process $X$.
Hence, under the conditions, $\Ex(x,r)$ plays the same role of the scaling function $\phi(x,r)$ in the previous subsection.

\item{\rm (ii)}
From $\VD$ and Assumption \ref{A1} (ii), (iii) and (iv), we can deduce that for every
$a\in(0,1]$ and $L>0$, there exists a constant $c_5:=c_{a,L,5}\ge1$
such that the following holds
for all $x, y\in M$ with $d(x,y)\le r$,
\be\label{eq:wondf1-1}
c_5^{-1}\Ca(B(y,aLr),B(y,2Lr))\le  \Ca(B(x,r), B(x,2r)) \le c_{ 5} \Ca(B(y,aLr),B(y,2Lr)).
\ee

\item{\rm (iii)} Assumption \ref{A1} is the same as  \cite[Assumption 1.6]{BM} except that
 in their paper the corresponding conditions
 are assumed to hold for $r\in (0, R_0]$ and for $0<r\le R\leq R_0$
with some $R_0>0$. These conditions are called
bounded geometry condition in \cite{BM}.
However the setting of \cite{BM} is for strongly local Dirichlet forms with underlying
state space $M$ being
geodesic. Under these settings and
the  bounded geometry condition, it is shown  in \cite{BM}
  that there exists an equivalent
doubling measure $\wt \mu$ on
$M$ so that Assumption \ref{A1}
holds (i.e., the bounded geometry condition holds globally
in large scale as well).
 Since harmonicity is invariant under
 time-changes by strictly increasing continuous additive functionals,
this enables them to substantially extend  the stability result of elliptic Harnack inequality of Bass \cite{B} for diffusions,
 which was essentially established under
the global bounded geometry condition.   However the continuity of the processes (i.e.\ diffusions) and the
geodesic property of the underlying state space played
a crucial role in \cite{BM}.
It is unclear at this stage whether
Assumption \ref{A1} can be replaced by a bounded geometry condition
 for  non-local Dirichlet forms on general metric measure spaces.
\end{description}
\end{remark}

The following result gives a  stable characterization of $\EHI$.

\begin{theorem}\label{C:1.12}
Under Assumptions $\ref{A1}$ and $\ref{A1-b}$,
 if $\J_{\Ex}$ holds, then
$$
\FK(\Ex)+\PI(\Ex)+\CSJ(\Ex)\Longleftrightarrow \WEHI^+(\Ex) \Longleftrightarrow\WEHI(\Ex) \Longleftrightarrow  \EHI(\Ex) \Longleftrightarrow \EHI,
$$ where $\J_\Ex$ is $\J_\phi$ with $\Ex(x,r)$ replacing $\phi(x,r)$, and same for other notions.
\end{theorem}

\subsection{Parabolic Harnack inequalities }\label{S:1.3}
As consequences of the main result of this paper, Theorem \ref{T:ehi-1} and the stability result of parabolic Harnack inequality in \cite[Theorem 1.17]{CKW2}, we will present in this subsection
new equivalent characterizations of parabolic Hanack inequality
in terms of elliptic Harnack inequalities.
In this subsection, we always assume that, for each $x\in M$ there is a kernel $J(x, dy)$ so that
$$
J(dx, dy) = J(x, dy)\, \mu (dy).
$$
We aim to present some equivalent conditions for parabolic Harnack inequalities in terms of elliptic Harnack inequalities, which
can be viewed as a complement to \cite{CKW2}.
We restrict ourselves to the case that the (scale) function $\phi$ is independent of $x$, i.e. in
this subsection,  $\phi: \bR_+\to \bR_+$ is a strictly increasing continuous
function  with $\phi (0)=0$, $\phi(1)=1$
such that there exist constants $c_3,c_4>0$ and $\beta_2\ge \beta_1>0$ so that
\be\label{pvdx}
 c_3 \Big(\frac Rr\Big)^{\beta_1} \leq
\frac{\phi (R)}{\phi (r)}  \ \leq \ c_4 \Big(\frac
Rr\Big)^{\beta_2}
\quad \hbox{for all }
0<r \le R.
\ee

We first give the probabilistic definition of parabolic
functions in the general context of metric measure spaces. Let $Z:=\{V_s,X_s\}_{s\ge0}$ be the space-time process corresponding to $X$ where $V_s=V_0-s$
for all $s\ge0$.
The filtration generated by $Z$ satisfying the usual conditions will be denoted by $\{\widetilde{\mathcal{F}}_s;s\ge0\}$. The law of the space-time process $s\mapsto Z_s$ starting from $(t,x)$ will be denoted by $\bP^{(t,x)}$. For every open subset $D$ of $[0,\infty)\times M$, define
$\tau_D=\inf\{s>0:Z_s\notin D\}.$ We say that a
nearly  Borel measurable function $u(t,x)$ on
$[0,\infty)\times M$ is \emph{parabolic} (or \emph{caloric}) in
$D=(a,b)\times B(x_0,r)$ for the process $X$ if there is a properly
exceptional set $\mathcal{N}_u$ of the process $X$ so that for every
relatively compact open subset $U$ of $D$,
$u(t,x)=\bE^{(t,x)}u(Z_{\tau_{U}})$ for every $(t,x)\in
U\cap([0,\infty)\times (M\backslash \mathcal{N}_u)).$

We next give definitions of parabolic Harnack inequality and
H\"older regularity for parabolic functions.

\begin{definition}\label{thm:defPHI} \rm \begin{description}
\item{(i)}
We say that \emph{parabolic Harnack inequality} $\PHI(\phi)$ holds for the process $X$, if there exist constants $0<C_1<C_2<C_3<C_4$,  $C_5>1$ and  $C_6>0$ such that for every $x_0 \in M $, $t_0\ge 0$, $R>0$ and for
every non-negative function $u=u(t,x)$ on $[0,\infty)\times M$ that is parabolic in
cylinder $Q(t_0, x_0,C_4\phi(R),C_5R):=(t_0, t_0+C_4\phi(R))\times B(x_0,C_5R)$,
$$
  \esssup_{Q_- }u\le C_6 \,\essinf_{Q_+}u,
$$ where $Q_-:=(t_0+C_1\phi(R),t_0+C_2\phi(R))\times B(x_0,R)$ and $Q_+:=(t_0+C_3\phi(R), t_0+C_4\phi(R))\times B(x_0,R)$.

\item{(ii)} We say
\emph{parabolic H\"older regularity}
$\PHR(\phi)$
holds for the process $X$, if there exist  constants $c>0$,
$\theta\in (0, 1]$ and $\eps \in (0, 1)$ such that for every $x_0 \in M $, $t_0\ge0$, $r>0$ and for
 every bounded measurable function $u=u(t,x)$ that is parabolic in
$Q(t_0,x_0,\phi(r), r)$,    there is a properly exceptional set ${\cal N}_u\supset {\cal N}$ so that
$$
|u(s,x)-u(t, y)|\le c\left( \frac{\phi^{-1}(|s-t|)+d(x, y)}{r} \right)^\theta  \esssup_{ [t_0, t_0+\phi (r)] \times M}|u|
$$
 for  every  $s, t\in (t_0, t_0+\phi (\eps r))$ and $x, y \in B(x_0, \eps r)\setminus {\cal N}_u$.
\end{description}
\end{definition}

\medskip

\begin{definition}\label{thm:defUJS} {\rm
We say that $\UJS$ holds if there is a symmetric function $J(x, y)$ so that
$J(x, dy)=J(x, y)\,\mu (dy)$, and there is a constant $c>0$ such that for $\mu$-a.e. $x, y\in M$ with $x\not= y$,
$$
J(x,y)\le  \frac{c}{V(x,r)}\int_{B(x,r)}J(z,y)\,\mu(dz)
\quad\hbox{for every }
0<r\le  \frac{1}{2} d(x,y).
$$}
\end{definition}
We define $\EHR$, $\E_\phi$, $\E_{\phi,\le}$, $\J_{\phi,\le}$, $\PI(\phi)$ and $\CSJ(\phi)$
similarly
as in previous subsections but with $\phi (r)$ in place of $\phi (x, r)$.
The following stability result of $\PHI (\phi)$ is recently established in \cite{CKW2}.

\begin{theorem}\label{T:PHI}{\bf (\cite[Theorem 1.17]{CKW2})}\,\,
Suppose that the metric measure space  $(M, d,  \mu)$ satisfies $\VD$ and $\RVD$, and $\phi$ satisfies \eqref{pvdx}.
Then the following are equivalent:
\begin{itemize}
\item[{\rm (i)}] $\PHI(\phi)$.
\item[{\rm (ii)}]
 $\PHR(\phi) + \E_{\phi,\le} + \UJS$.
 \item[{\rm (iii)}] $\EHR + \E_\phi + \UJS$.
\item[{\rm (iv)}] $\J_{\phi,\le} + \PI(\phi) + \CSJ(\phi) +\UJS$.
\end{itemize}
\end{theorem}

As a consequence of Theorems \ref{T:ehi-1} and \ref{T:PHI}, we have the following statement for the equivalence of $\PHI(\phi)$ in terms of $\EHI$.

\begin{theorem}\label{T:ehi} Suppose that the metric measure space  $(M, d,  \mu)$ satisfies $\VD$ and $\RVD$, and $\phi$ satisfies \eqref{pvdx}. Then
the following are equivalent:
\begin{itemize}
\item[{\rm (i)}] $\PHI(\phi)$.
\item[{\rm (ii)}] $\WEHI^+(\phi)+ \E_\phi+\UJS.$
\item[{\rm (iii)}] $\WEHI(\phi)+ \E_\phi+\UJS.$
\item[{\rm (iv)}] $\EHI(\phi)+ \E_\phi+\UJS.$
\item[{\rm (v)}] $\EHI+ \E_\phi+\UJS+\J_{\phi,\le}.$
\end{itemize}
 \end{theorem}

\begin{proof} As  indicated in Theorem \ref{T:PHI}, under $\VD$, $\RVD$ and \eqref{pvdx}, $$ \PHI(\phi)\Longleftrightarrow\J_{\phi,\le}+\PI(\phi)+
\CSJ(\phi)+\UJS\Longrightarrow \E_\phi.$$  Then, by Theorem \ref{T:ehi-1}(ii), (i) $\Longrightarrow$ (ii). (ii) $\Longrightarrow$ (iii) is clear. (iii) $\Longrightarrow$ (i) follows from Theorem \ref{T:ehi-1}(i) and Theorem \ref{T:PHI}(iii).

Obviously, (i) $\Longrightarrow$ (v) is a consequence of Theorem \ref{T:PHI} (i), (iii) and (iv).  (v) $\Longrightarrow$ (iv) follows from Theorem \ref{T:ehi-1}(iii). (iv) $\Longrightarrow$ (iii) is trivial. This completes the proof.
\qed\end{proof}

\bigskip

The remainder of this paper is mainly concerned with the proof of
Theorem \ref{T:ehi-1}, the main result of this paper. It is organized as follows.
The proofs of  Theorem \ref{T:ehi-1}(i), (ii) and (iii) are given in the next three sections, respectively.
In Section \ref{S:5}, we study the relations between the mean of exit time and
relative capacity.
In particular, the proof of Theorem \ref{C:1.12} is given there.
Finally,   a class of symmetric jump processes of variable orders on $\bR^d$ with state-dependent scaling functions are given in Section \ref{S:6},
for which we apply  the main results of this paper to show that all the elliptic Harnack inequalities hold for these processes.

\medskip

In this paper, we use ``:="' as a way of definition.
For two functions $f$ and $g$, notation $f\asymp g$ means that there is a constant $c\geq 1$ so that
$g/c \leq f \leq c g$.

\section{Elliptic Harnack inequalities and H\"older regularity}
\label{S:2}

In this section, we assume that $\mu$ and $\phi$ satisfy $\VD$, \eqref{pvd1} and \eqref{pvd}, respectively.
We will prove that $\WEHI(\phi)$ implies a priori H\"older regularity for harmonic functions,
and study
the relation between $\WEHI(\phi)$ and $\EHI (\phi)$.

\subsection{$\WEHI(\phi)\Longrightarrow \EHR$}\label{S:2.1}

In this part, we will show that the weak elliptic Harnack inequality implies regularity
estimates of harmonic functions in H\"{o}lder spaces. We mainly follow the strategy of \cite[Theorem 1.4]{DK}, part of which is originally due to \cite{Mos,Sil}.

\begin{theorem}\label{P:holder}
Suppose that $\VD$, \eqref{pvd1} and $\WEHI(\phi)$ hold. Then there
exist constants $\beta\in(0,1)$ and $c>0$ such that for any $x_0\in
M$, $r>0$ and harmonic function $u$ on $B(x_0,r)$,
\begin{equation}\label{ee:ewehi}\essosc_{B(x_0,\rho)}u\le c \,\|u\|_\infty \cdot \left(
\frac{\rho}{r}\right)^\beta,\quad 0<\rho\le r.\end{equation} In
particular, $\EHR$ holds.
\end{theorem}

\begin{proof} (1) Without loss of generality, we  assume   the harmonic  function $u$ is bounded.  Throughout the proof, we fix $x_0\in M$, and denote by $B_r=B(x_0,r)$ for any $r>0$. For a given bounded harmonic  function $u$ on $B_r$, we will
construct an increasing sequence $(m_n)_{n\ge1}$ of positive numbers and a decreasing sequence $(M_n)_{n\ge1}$ that satisfy for any $n\in \bN\cup \{0\}$,
\begin{equation}\label{e:profh1}
\begin{split}&m_n\le u(x)\le M_n \quad \hbox{ for } x\in B_{r\theta^{-n}};\\
&M_n-m_n=K\theta^{-n\beta}.\end{split}
\end{equation}
Here $K=M_0-m_0\in[0,2 \|u\|_\infty]$ with $M_0=\|u\|_\infty$ and $m_0=\essinf_{M} u$,  and the constants $\theta=\theta(\delta)\ge \delta^{-1}$ and $\beta=\beta(\delta)\in(0,1)$ are determined later so  that
\begin{equation}\label{e:constantse}
\frac{2-\lambda}{2}\theta^\beta\le 1 \quad \hbox{for } \lambda :=(2^{1+1/\varepsilon}c)^{-1}\in(0,1),\end{equation}
where $\eps, \delta \in (0, 1)$ and $c\geq 1$   are the constants in the definition of $\WEHI(\phi)$.

Let us first show that how this construction proves the first
desired assertion \eqref{ee:ewehi}. Given $\rho <r$, there is a
$j\in \bN\cup \{0\}$ such that
$$r\theta^{-j-1}\le \rho<r\theta^{-j}.$$ From \eqref{e:profh1}, we conclude
$$\essosc_{B_\rho}u\le \essosc_{B_{r\theta^{-j}}}u\le M_j-m_j=  K \theta^{-j\beta}\le 2 \theta^\beta\|u\|_\infty\left(\frac{\rho}{r}\right)^\beta.$$

Set $M_{-n}=M_0$ and $m_{-n}=m_0$ for any $n\in \bN$. Assume that there is a $k\in \bN$ and there are $M_n$ and $m_n$ such that \eqref{e:profh1} holds for $n\le k-1$. We need to choose $m_k$, $M_k$ such that \eqref{e:profh1} still holds for $n=k$. Then the desired assertion follows by induction. For any $x\in M$, set
$$v(x)= \left(u(x)-\frac{M_{k-1}+m_{k-1}}{2}\right)\frac{2\theta^{(k-1)\beta}}{K}.$$ Then the definition of $v$ implies that $|v(x)|\le 1$ for almost all $x\in B_{r\theta^{-(k-1)}}$. Given $y\in M$ with $d(y,x_0)\ge r\theta^{-(k-1)}$, there is a $j\in \bN$ such that
$$r\theta^{-k+j}\le d(y,x_0)<r\theta^{-k+j+1}.$$ For such $y\in M$ and $j\in \bN$, on the one hand, we conclude that
\begin{align*}\frac{K}{2\theta^{(k-1)\beta}}v(y)=&u(y)-\frac{M_{k-1}+m_{k-1}}{2}\\
\le& M_{k-j-1}-m_{k-j-1}+m_{k-j-1}- \frac{M_{k-1}+m_{k-1}}{2}\\
\le& M_{k-j-1}-m_{k-j-1}-\frac{M_{k-1}-m_{k-1}}{2}\\
\le &K \theta^{-(k-j-1)\beta}-\frac{K}{2}\theta^{-(k-1)\beta},\end{align*} where in the equalities above we used the fact that
if $j> k-1$, then
$u(y)\le M_0$, $m_{k-j-1}\ge m_0$ and $M_0-m_0\le
K\theta^{-(k-j-1)\beta}$.
That is,
\begin{equation}\label{e:lowerph}v(y)\le 2 \theta^{j\beta}-1\le 2\left(\frac{d(y,x_0)}{r\theta^{-k}}\right)^\beta-1.\end{equation} On the other hand, similarly, we have
\begin{align*}\frac{K}{2\theta^{(k-1)\beta}}v(y)=&u(y)-\frac{M_{k-1}+m_{k-1}}{2}\\
\ge& m_{k-j-1}-M_{k-j-1}+M_{k-j-1}- \frac{M_{k-1}+m_{k-1}}{2}\\
\ge& -(M_{k-j-1}-m_{k-j-1})+\frac{M_{k-1}-m_{k-1}}{2}\\
\ge &-K \theta^{-(k-j-1)\beta}+\frac{K}{2}\theta^{-(k-1)\beta},\end{align*}
i.e.
$$v(y)\ge 1- 2\theta^{j\beta}\ge 1-2\left( \frac{d(y,x_0)}{r\theta^{-k}}\right)^\beta.$$

Now, there are two cases:

(i) $\displaystyle \mu \left( \{x\in B_{r\theta^{-k}}: v(x)\le 0\} \right) \ge  \mu( B_{r\theta^{-k}}) /2$.

(ii) $\displaystyle \mu \left( \{x\in B_{r\theta^{-k}}: v(x)> 0\} \right)\ge  \mu( B_{r\theta^{-k}})/2$.

In case (i) we aim to show $v(z)\le 1-\lambda$ for almost every $z\in B_{r\theta^{-k}}$. If this holds true, then for any $z\in B_{r\theta^{-k}}$,
\begin{align*}u(z)&\le \frac{(1-\lambda)K}{2}\theta^{-(k-1)\beta}+\frac{M_{k-1}+m_{k-1}}{2}\\
&=\frac{(1-\lambda)K}{2}\theta^{-(k-1)\beta}+\frac{M_{k-1}-m_{k-1}}{2}+m_{k-1}\\
&=\frac{(1-\lambda)K}{2}\theta^{-(k-1)\beta}+\frac{K}{2}\theta^{-(k-1)\beta}+m_{k-1}\\
&\le K\theta^{-k\beta}+m_{k-1}, \end{align*} where the last inequality follows from the first inequality in \eqref{e:constantse}. Thus, we set $m_k=m_{k-1}$ and $M_k=m_k+K\theta^{-k\beta}$, and obtain that $m_k\le u(z)\le M_k$ for almost every $z\in B_{r\theta^{-k}}.$

Consider $w=1-v$ and note that $w\ge 0$ in $B_{r\theta^{-(k-1)}}$.
Since in the present setting there is no killing inside $M_0$ for the process $X$,
constant functions are
harmonic, and so $w$ is also harmonic function. Applying $\WEHI(\phi)$ with $w$ on $B_{r\theta^{-(k-1)}}$, we find that
\be\label{eq:fenon1}\begin{split}
\bigg(\frac{1}{\mu (B_{r\theta^{-k}})}&\int_{B_{r\theta^{-k}}} w^\varepsilon\,du\bigg)^{1/\varepsilon}\\
&\le c_1\left(\essinf _{B_{r\theta^{-k}}}w+{\phi(x_0,r\theta^{-k})}{\T_\phi}(w_-; x_0, r\theta^{-(k-1)})\right).\end{split}
\ee
Note that, since the constant $c$ in the definition of $\WEHI(\phi)$ may depend on $\delta$ and $\varepsilon$, in the above inequality the constant $c_1=c$ could also depend on $\delta$ and $\varepsilon$, thanks to the fact that $\theta^{-1}\le \delta$. Under case (i),
\be\label{eq:fenon2}
\left(\frac{1}{\mu (B_{r\theta^{-k}})}\int_{B_{r\theta^{-k}}} w^\varepsilon\,du\right)^{1/\varepsilon}\ge 2^{-1/\varepsilon}.\ee
On the other hand, by \eqref{e:lowerph}, Remark \ref{intelem} and \eqref{pvd1},
\begin{equation}\label{eq:fenon3}\begin{split}\phi&(x_0,r\theta^{-(k-1)}){\T_\phi}(w_-; x_0, r\theta^{-(k-1)})\\
\le& \phi(x_0,r\theta^{-(k-1)})\int_{B_{r\theta^{-(k-1)}}^c} \frac{(1-v(z))_-}{V(x_0,d(x_0,z))\phi(x_0,d(x_0,z))}\,\mu(dz)\\
\le&\phi(x_0,r\theta^{-(k-1)})\sum_{j=1}^\infty\int_{B_{r\theta^{-k+j+1}}\backslash B_{r\theta^{-k+j}}} \frac{(1-v(z))_-}{V(x_0,d(x_0,z))\phi(x_0,d(x_0,z))}\,\mu(dz)\\
\le&\phi(x_0,r\theta^{-(k-1)})\sum_{j=1}^\infty\int_{B_{r\theta^{-k+j+1}}\backslash B_{r\theta^{-k+j}}} \frac{(v(z)-1))_+}{V(x_0,d(x_0,z))\phi(x_0,d(x_0,z))}\,\mu(dz)\\
\le& 2\phi(x_0,r\theta^{-(k-1)})\\
&\times \sum_{j=1}^\infty \int_{B_{r\theta^{-k+j+1}}\backslash B_{r\theta^{-k+j}}} \left[ \left(\frac{ d(x_0,z)}{r\theta^{-k}}\right)^\beta-1\right]\frac{1}{V(x_0,d(x_0,z))\phi(x_0,d(x_0,z))}\,\mu(dz)\\
\le&c_2 \phi(x_0,r\theta^{-(k-1)})\sum_{j=1}^\infty\frac{\theta^{(j+1)\beta}-1}{\phi(x_0,r\theta^{-k+j})}\\
\le& c_3 \sum_{j=1}^\infty\theta^{-j\beta_1}(\theta^{j\beta}-1),\end{split}\end{equation}
where $c_3>0$ is a constant independent of $k$ and $r$ but depend on $\theta$ and $\beta_1$ from \eqref{pvd1}.
Hence,
by \eqref{pvd1}, \eqref{eq:fenon1}, \eqref{eq:fenon2} and
\eqref{eq:fenon3}, we obtain
\begin{align*}\essinf _{B_{r\theta^{-k}}}w&\ge (c_12^{1/\varepsilon})^{-1}-c_4 \phi(x_0,r\theta^{-(k-1)}){\T_\phi}(w_-; x_0, r\theta^{-(k-1)})\\
&\ge (c2^{1/\varepsilon})^{-1}- c_5 \sum_{j=1}^\infty\theta^{-j\beta_1}(\theta^{j\beta}-1). \end{align*} Note that all the constants $c_i$ $(i=1,\ldots, 5)$ may depend on $\theta$. Since for any $\beta\in(0,\beta_1)$,
$$  \sum_{j=1}^\infty\theta^{-j\beta_1}(\theta^{j\beta}-1)<\infty,$$ we can choose $l$ large enough (which is independent of $\beta,\theta$ and only depends on $\delta$) such that for any $\beta\in (0,\beta_1/2)$,
$$ \sum_{j=l+1}^\infty\theta^{-j\beta_1}(\theta^{j\beta}-1)\le \sum_{j=l+1}^\infty\theta^{-j\beta_1}(\theta^{j\beta_1/2}-1)\le \sum_{j=l+1}^\infty\delta^{j\beta_1/2}< (4c_5c2^{1/\varepsilon})^{-1}.$$ Given $l$, one can further take $\beta\in(0,\beta_1/2)$ small enough such that $$\sum_{j=1}^l\theta^{-j\beta_1}(\theta^{j\beta}-1)\le \beta (\log \theta)\sum_{j=1}^l\theta^{-j(\beta_1-\beta)} j \le \beta l \theta^{-\beta_1/2} (\log \theta)< (4c_5c2^{1/\varepsilon})^{-1}.$$ (Without loss of generality we may and do assume that $\delta$ in the definition of $\WEHI(\phi)$ is small enough. Thus, the constant $\beta$ here is also independent of $\theta$ and only depends on $\delta$.)  Therefore,
$$\essinf _{B_{r\theta^{-k}}}w\ge (2c2^{1/\varepsilon})^{-1}=\lambda.$$ That is, $v\le 1-\lambda$ on $B_{r\theta^{-k}}$.

In case (ii), our aim is to show $v\ge -1+\lambda$. This time we set $w=1+v$. Following the arguments above, one sets $M_k=M_{k-1}$ and $m_k=M_k-K\theta^{-k\beta}$ leading to the desired result.

(2) Let $\delta_0\in (0,1/3)$. Then for almost all $x,y\in B(x_0,\delta r)$, the function $u$ is harmonic on $B(x,(1-\delta_0)r)$. Note that $d(x,y)\le 2\delta_0r\le (1-\delta_0)r$. Applying \eqref{ee:ewehi}, we have
$$|u(x)-u(y)|\le \essosc_{B(x,d(x,y))}u\le c\|u\|_\infty\cdot \left(\frac{d(x,y)}{(1-\delta_0)r}\right)^\beta.
$$
This establishes $\EHR$. \qed
\end{proof}

\begin{remark}\rm The argument above in fact shows that   $\WEHI(\phi)\Longrightarrow \EHR$ holds for any general jump processes (possibly non-symmetric) that admits no killings inside $M$.
\end{remark}

\subsection{$\WEHI(\phi)+\J_\phi+\FK(\phi)+\CSJ(\phi)\Longrightarrow \EHI(\phi)$}\label{S:2.2}

Let $D$ be an open subset of $M$.
Recall that a function $f$ is said to be locally in $\sF_{D}$, denoted as $f\in \sF_{D}^{loc}$, if for every relatively compact subset $U$ of ${D}$, there is a function $g\in \sF_{D}$ such that $f=g$ $m$-a.e. on $U$.
The following is established in    \cite{Chen}.
\begin{lemma}{\rm (\cite[Lemma 2.6]{Chen})}
Let ${D}$ be an open subset of $M$. Suppose $u$ is a function in $\sF_{D}^{loc}$ that is locally bounded on ${D}$
and satisfies that
\begin{equation}\label{con-1}
\int_{U\times V^c} |u(y)|\,J(dx,dy)<\infty
\end{equation} for any relatively compact open sets $U$ and $V$ of $M$ with $\bar{U}\subset V \subset \bar{V} \subset {D}$.
 Then  for every $v\in \sF\cap C_c({D})$, the expression
$$
\int (u(x)-u(y))(v(x)-v(y))\,J(dx,dy)
$$ is well defined and finite; it will still be denoted as $\sE(u,v)$.
\end{lemma}
As noted in \cite[(2.3)]{Chen}, since $(\sE,\sF)$ is a regular Dirichlet form on $L^2(M; \mu)$, for any relatively compact open sets $U$ and $V$ with $\bar{U}\subset V$, there is a function $\psi\in \sF\cap C_c(M)$ such that $\psi=1$ on $U$ and $\psi=0$ on
$V^c$.
 Consequently,
$$\int_{U\times V^c}\,J(dx,dy)=\int_{U\times V^c}(\psi(x)-\psi(y))^2\,J(dx,dy)\le \sE(\psi,\psi)<\infty,$$
so each bounded function $u$ satisfies \eqref{con-1}.

We say that a nearly Borel measurable function $u$ on $M$
is \emph{$\sE$-subharmonic}  (resp. \emph{$\sE$-harmonic, $\sE$-superharmonic})
in ${D}$ if $u\in\sF_{D}^{loc}$ that is locally bounded on $D$, satisfies
\eqref{con-1}
for any relatively compact open sets $U$ and $V$ of $M$ with $\bar{U}\subset V \subset \bar{V} \subset {D}$, and that
\begin{equation*}\label{an-har}
\sE(u,\varphi)\le 0 \quad (\textrm{resp.}\ =0, \ge0)
\quad \hbox{for any } 0\le\varphi\in\sF\cap C_c(D).
\end{equation*}

The following is established in \cite[Theorem 2.11  and Lemma 2.3]{Chen} first for harmonic functions,
and then extended in \cite[Theorem 2.9]{ChK} to subharmonic functions.

\begin{theorem}\label{equ-har} Let ${D}$ be an open subset of $M$, and   $u$ be  a bounded function.
Then $u$ is $\sE$-harmonic $($resp.  $\sE$-subharmonic$)$ in ${D}$ if and only if $u$ is  harmonic
 $($resp. subharmonic$)$
 in ${D}$.
 \end{theorem}

The next lemma
can be proved by the same argument as that for
\cite[Proposition 2.3]{CKW1}.

\begin{lemma}    \label{cap-1}
Assume that $\VD$, \eqref{pvd1}, \eqref{pvd}, $\J_{\phi,\le}$ and $\CSJ(\phi)$ hold. Then there is a constant $c_0>0$ such that for  every
$0<r\le R$ and almost all $x\in M$,
$$
\mbox{\rm Cap} (B(x,R),B(x,R+r))\le c_0\frac{V(x,R+r)}{\phi(x,r)}.
$$\end{lemma}

Using this lemma, we can establish
 the following.

\begin{lemma}\label{L:cont} Let $B_r=B(x_0,r)$ for some $x_0\in M$ and $r>0$. Assume that $u$ is a bounded and $\sE$-superharmonic function on $B_R$ such that $u\ge 0$ on $B_R$.  If $\VD$, \eqref{pvd1}, \eqref{pvd}, $\J_\phi$, $\FK(\phi)$ and $\CSJ(\phi)$
hold, then for any $0<r<R$,
$$\phi(x_0,r){\T_\phi}\, (u_+; x_0, r)\le c\left(\esssup_{B_r}u+ {\phi(x_0,r)}{\T_\phi}\,(u_-; x_0, R)\right),$$ where $c>0$ is a constant independent of $u$, $x_0$, $r$ and $R$.
\end{lemma}

\begin{proof} According to $\J_{\phi,\le}$, $\CSJ(\phi)$ and Lemma \ref{cap-1}, we can choose $\varphi\in \sF_{B_{3r/4}}$ related
to $\mbox{Cap} (B_{r/2},B_{3r/4})$ such that
\begin{equation}\label{e:u01}
\sE(\varphi,\varphi)\le 2  \mbox{Cap} (B_{r/2},B_{3r/4}) \le
\frac{c_1V(x_0,r)}{\phi(x_0,r)}.
\end{equation} Let $k=\esssup_{B_r}u$ and $w=u-2k$.
Since $u$ is an $\sE$-superharmonic function on $B_R$,
and $w\varphi^2\in \sF_{B_{3r/4}}$ with $w<0$ on $B_r$,
\begin{align*}0\ge \sE\big(u,w\varphi^2\big)
=&\int_{B_{r}\times B_{r}} (u(x)-u(y))(
w(x)\varphi^2(x)-w(y)\varphi^2(y))\,J(dx,dy)\\
&+2\int_{B_{r}\times B_{r}^c} (u(x)-u(y))
w(x)\varphi^2(x)\,J(dx,dy)\\
=&:I_1+2I_2.\end{align*}

For any $x,y\in B_r$, \begin{align*}
&(u(x)-u(y))(w(x)\vp^2(x)-w(y)\vp^2(y))\\
&=(w(x)-w(y))(w(x)\vp^2(x)-w(y)\vp^2(y))\\
&=\vp^2(x)(w(x)-w(y))^2+w(y)(\vp^2(x)-\vp^2(y))(w(x)-w(y))\\
&\ge \vp^2(x)(w(x)-w(y))^2 -\frac{1}{8} (\vp(x)+\vp(y))^2(w(x)-w(y))^2-2w^2(y)(\vp(x)-\vp(y))^2,
\end{align*}
where we used the fact that $ab\ge -\big( \frac{1}{8}a^2+2b^2\big)$
for all $a,b\in \bR$ in the inequality above. Hence,
\begin{align*}I_1\ge&\int_{B_r\times B_r} \vp^2(x)(w(x)-w(y))^2\,J(dx,dy)\\
&-\frac{1}{8}\int_{B_r\times B_r}(\vp(x)+\vp(y))^2(w(x)-w(y))^2\,J(dx,dy)\\
&- 2\int_{B_r\times B_r} w^2(y)(\vp(x)-\vp(y))^2\, J(dx,dy)\\
\ge& \frac 12\int_{B_r\times B_r}
\vp^2(x)(w(x)-w(y))^2\,J(dx,dy)\\
&-8k^2\int_{B_r\times B_r}(\vp(x)-\vp(y))^2\, J(dx,dy)\\
\ge&-8k^2\int_{B_r\times B_r}(\vp(x)-\vp(y))^2\, J(dx,dy) , \end{align*}
where in the second inequality we have used the symmetry property of $J(dx,dy)$ and the fact that $w^2\le 4k^2$ on $B_r$.

On the other hand, by the definition of $w$, it is easy to see that for any $x\in B_r$ and $y\notin B_r$
\begin{align*}
(u(x)-u(y))w(x) &\ge k(u(y)-k)_+-2k{\bf 1}_{\{u(y)\le k\}} (u(x)-u(y))_+\\
&\ge k(u(y)-k)_+-2k (u(x)-u(y))_+,
\end{align*}
and so
\begin{align*}I_2\ge& \int_{B_r\times B_r^c}
k(u(y)-k)_+\vp^2(x)\,J(dx,dy)\\
&-\int_{B_r\times B_r^c}2k(u(x)-u(y))_+\vp^2(x)\,J(dx,dy)\\
=&:I_{21}-I_{22}.\end{align*}
 Furthermore, since $(u(y)-k)_+\ge u_+(y)-k$, we find that
 \begin{align*}I_{21}\ge
 &k\int_{B_r\times B_r^c}u_+(y)\vp^2(x)\,J(dx,dy)-k^2\int_{B_r\times B_r^c}
 \vp^2(x)\,J(dx,dy)\\
 \ge& k\mu(B_{r/2})\inf_{x\in B_{r/2}}\int_{B_r^c}u_+(y)\,J(x,dy)-k^2\int_{B_r\times B_r^c}
 \vp^2(x)\,J(dx,dy)\\
 \ge& {c_1}k{V(x_0,r)}{\T_\phi}\,(u_+;x_0,r)
 -k^2\int_{B_r\times B_r^c}\vp^2(x)\,J(dx,dy),
 \end{align*}
 where in the second inequality we have used the fact that $\varphi=1$ on $B_{r/2}$, and in the last inequality we have used $\J_{\phi,\ge}$ and the fact that for all $x\in B_{r/2}$ and $z\in B_{r}^c$,
 $$\frac{V(x,d(x,z))}{V(x_0, d(x_0,z))}
 \frac{\phi(x,d(x,z))}{\phi(x_0,d(x_0,z))}\le c'\left(1+ \frac{d(x,x_0)}{d(x_0,z)}\right)^{d_2+\beta_2}\le c'',$$
 thanks to $\VD$, \eqref{pvd1} and \eqref{pvd}.
 Also, since $u\ge0$ on $B_R$, we can check that
 \begin{align*}I_{22}\le&  2k
 \int_{B_r\times (B_R\backslash B_r)} k\vp^2(x)\,J(dx,dy)+2k\int_{B_r\times B_R^c}
 (k+u_-(y))\vp^2(x)\,J(dx,dy)\\
 \le &2k^2\int_{B_r\times B_r^c}
 \vp^2(x)\, J(dx,dy)+  c_2k^2 \frac{V(x_0,r)}{\phi(x_0,r)}+ c_2k{V(x_0,r)}{\T_\phi}\, (u_-; x_0,R),\end{align*} where the second term of the last inequality follows from Remark \ref{intelem} and \eqref{pvd}, and in the third term we have used $\J_{\phi,\le}$.

By the estimates for $I_{21}$ and $I_{22}$, we get that
\begin{align*}I_2\ge &-3k^2\int_{B_r\times B_r^c}\vp^2(x)
\, J(dx,dy) +{c_1}k{V(x_0,r)}{\T_\phi}\,(u_+;x_0,r)\\
 &- c_2k^2 \frac{V(x_0,r)}{\phi(x_0,r)}-c_2k{V(x_0,r)}{\T_\phi}\, (u_-; x_0,R).\end{align*} This along with the estimate for $I_1$ yields that
 $${V(x_0,r)}{\T_\phi}\,(u_+; x_0,r)\le c_3\left[k\bigg(\frac{V(x_0,r)}{\phi(x_0,r)}+ \sE(\vp,\vp)\bigg)+ {V(x_0,r)} {\T_\phi}\,(u_-; x_0,R)\right].$$ Then, combining this inequality with  \eqref{e:u01} proves the desired assertion.
 \qed\end{proof}

We also need the following result. Since the proof is essentially the same as that of \cite[Proposition 4.10]{CKW1}, we omit it here.
\begin{proposition} \label{P:mvi2g}
Let $x_0\in M$ and $R>0$.
Assume $\VD$, \eqref{pvd1}, \eqref{pvd}, $\J_{\phi,\le}$, $\FK(\phi)$ and $\CSJ(\phi)$
 hold, and let $u$
be a  bounded
$\sE$-subharmonic
in $B(x_0,R)$.  Then  for any $\delta>0$,
$$
 \esssup_{B(x_0,R/2)} u
 \le  c_1\left[ \left(\frac{(1+\delta^{-1})^{1/\nu}}{V(x_0,R)}\int_{B(x_0,{R})} u^2\,d\mu \right)^{1/2}+\delta\phi(x_0,R)
 {\T_\phi}\, (u; x_0,R/2)
  \right] ,
$$
 where $\nu$ is the constant in $\FK(\phi)$, and $c_1>0$ is a constant independent of
$x_0$, $R$,
$\delta$ and $u$.
\end{proposition}

 We are in a position to present the main statement in this subsection.

\begin{theorem} Let $B_r(x_0)=B(x_0,r)$ for some $x_0\in M$ and $r>0$.
Assume that $u$ is a bounded and  $\sE$-harmonic
function on $B_R(x_0)$ such
that $u\ge 0$ on $B_R(x_0)$.  Assume that $\VD$, \eqref{pvd1}, \eqref{pvd}, $\J_{\phi}$, $\FK(\phi)$ and $\CSJ(\phi)$,
and $\WEHI(\phi)$ hold.  Then the following estimate holds for any $0<r<\delta_0R$,
$$\esssup_{B_{r/2}(x_0)}u\le c\left(\essinf_{B_r(x_0)}u +{\phi(x_0,r)} {\T_\phi} \, (u_-; x_0,R)\right),$$ where $\delta_0\in (0,1)$ is the constant $\delta$ in $\WEHI(\phi)$ and $c>0$ is a constant independent of $x_0$, $r,$ $R$ and $u$.  This is,
$$\WEHI(\phi)+\J_\phi+\FK(\phi)+\CSJ(\phi)\Longrightarrow \EHI(\phi).$$
\end{theorem}

\begin{proof} Note that $u_+$ is a bounded and $\sE$-subharmonic function on $B_R(x_0)$.
According to Proposition \ref{P:mvi2g}, for any $0<\delta<1$ and $0<\rho<R$,
$$\esssup_{B_{\rho/2}(x_0)}u\le c_1\left[\delta \phi(x_0,\rho){\T_\phi}\,(u_+;x_0, \rho/2)+ \delta^{-1/(2\nu)} \left(\frac{1}{V(x_0,\rho)}\int_{B_\rho(x_0)}u_+^2\,d\mu\right)^{1/2}\right],$$ where $c_1>0$ is a constant independent of $x_0$, $\rho$, $u$ and $\delta$.
The inequality above along with Lemma \ref{L:cont} yields that
\begin{align*}
\esssup_{B_{\rho/2}(x_0)}u \le c_2  \bigg(
 &\delta^{-1/(2\nu)} \left(\frac{1}{V(x_0,\rho)}\int_{B_\rho(x_0)}u_+^2\,d\mu\right)^{1/2} \\
&+\delta \esssup_{B_\rho(x_0)}u+ \delta {\phi(x_0,\rho)} {\T_\phi}\,(u_-;x_0, R) \bigg).
\end{align*}

For any $1/2\le \sigma'\le \sigma\le 1$ and $z\in B_{\sigma'r}(x_0)$, applying the inequality above with $B_{\rho}(x_0)= B_{(\sigma-\sigma')r}(z)$, we get that there is a constant $c_3>1$ such that
\begin{align*}u(z)\le
c_3\bigg( &\frac{\delta^{-1/(2\nu)}}{(\sigma-\sigma')^{d_2/2}}
\left(\frac{1}{V(x_0, \sigma r)}\int_{B_{\sigma r}(x_0)} u^2\,d\mu\right)^{1/2} \\
&+\delta \esssup_{B_{\sigma r}(x_0)}u+ \delta {\phi(x_0,r)} {\T_\phi}\,(u_-;x_0, R)\bigg),
\end{align*}
 where we have used the facts that
$B_{(\sigma-\sigma')r}(z)\subset B_{\sigma r}(x_0)$ for any $z\in B_{\sigma 'r}(x_0)$, and
$$
\frac{V(x_0,\sigma r)}{V(z,(\sigma -\sigma')r)}\le c'\left(1+\frac{d(x_0,z)+\sigma r}{(\sigma-\sigma')r}\right)^{d_2}\le c''\left(1+\frac{\sigma r+\sigma'r}{(\sigma-\sigma')r}\right)^{d_2}\le \frac{c'''}{(\sigma-\sigma')^{d_2}},$$ thanks to $\VD$.
Therefore,
\begin{align*}
\esssup_{B_{\sigma'r}(x_0)}u \le
c_3\bigg(  &\frac{\delta^{-1/(2\nu)}}{(\sigma-\sigma')^{d_2/2}}
\left(\frac{1}{V(x_0, \sigma r)}\int_{B_{\sigma r}(x_0)} u^2\,d\mu\right)^{1/2}\\
& +\delta \esssup_{B_{\sigma r}(x_0)}u+ \delta {\phi(x_0,r)} {\T_\phi}\,(u_-;x_0, R)\bigg).
\end{align*}
In particular, choosing $\delta=\frac{1}{4c_3}$ in the inequality above, we arrive at
\begin{align*}\esssup_{B_{\sigma'r}(x_0)}u \le &\frac{1}{4}\esssup_{B_{\sigma r}(x_0)}u\\
&+\frac{c_4}{(\sigma-\sigma')^{d_2/2}}\left(\frac{1}{V(x_0, \sigma r)}\int_{B_{\sigma r}(x_0)} u^2\,d\mu\right)^{1/2}+c_4 {\phi(x_0,r)} {\T_\phi}\,(u_-;x_0, R).\end{align*}
Since
\begin{align*}
&
\frac{c_4}{(\sigma-\sigma')^{d_2/2}}\left(\frac{1}{V(x_0, \sigma r)}\int_{B_{\sigma r}(x_0)} u^2\,d\mu\right)^{1/2}\\
&\le \frac{c_4}{(\sigma-\sigma')^{d_2/2}}\frac{(\esssup_{B_{\sigma r}(x_0)}u)^{(2-q)/2}}{V(x_0,\sigma r)^{1/2}}
\Big(\int_{B_{\sigma r}(x_0)}
|u|^q\,d\mu\Big)^{1/2}\\
&\le\frac 14 \esssup_{B_{\sigma r}(x_0)}u+
\frac{c_4'}{(\sigma-\sigma')^{d_2/q}}\Big(
\frac{1}{V(x_0,\sigma r)}\int_{B_{\sigma r}(x_0)}
|u|^q\,d\mu \Big)^{1/q},
\end{align*}
where in the last inequality we applied
the standard Young
inequality with exponent $2/q$ and $2/(2-q)$ with any $0<q<2$, we
have for any $0< q<2$ and $1/2\le \sigma'\le \sigma\le 1$,
\begin{align*} \esssup_{B_{\sigma 'r}(x_0)}u &\le  \frac{1}{2} \esssup_{B_{\sigma r}(x_0)}u\\
&\quad+
\frac{c_5}{(\sigma -\sigma')^{d_2/q}}
\bigg[\left(\frac{1}{V(x_0,\sigma r)}\int_{B_{\sigma r}(x_0)}
u^q\,d\mu\right)^{1/q} +{\phi(x_0,r)} {\T_\phi}\,(u_-;x_0, R)\bigg]\\
&\le \frac{1}{2} \esssup_{B_{\sigma r}(x_0)}u\\
&\quad+
\frac{c_5'}{(\sigma -\sigma')^{d_2/q}}
\bigg[\left(\frac{1}{V(x_0, r)}\int_{B_{r}(x_0)}
u^q\,d\mu\right)^{1/q} +{\phi(x_0,r)} {\T_\phi}\,(u_-;x_0, R)\bigg].\end{align*}
According to Lemma \ref{iter} below, we find that
$$\esssup_{B_{r/2}(x_0)}u \le c_6\bigg[\left(\frac{1}{V(x_0,r)}\int_{B_{r}(x_0)}
u^q\,d\mu\right)^{1/q} +{\phi(x_0,r)} {\T_\phi}\,(u_-;x_0, R)\bigg].$$ To conclude the proof, we combine the above inequality with $\WEHI(\phi)$ and Theorem \ref{equ-har}, by setting $q= \eps$. \qed
\end{proof}

The following lemma is taken from \cite[Lemma 1.1]{GG}, which has been used in the proof above.

\begin{lemma}\label{iter} Let $f(t)$ be a non-negative bounded function defined for $0\le T_0\le t\le T_1$. Suppose that for $T_0\le t\le s\le T_1$ we have
$$f(t)\le A(s-t)^{-\alpha}+B+\theta f(s),$$ where $A,B,\alpha,\theta$ are non-negative constants, and $\theta<1$. Then there exists a positive constant $c$ depending only on $\alpha$ and
$\theta$ such that for every $T_0\le r\le R\le T_1$, we have
$$f(r)\le c\Big( A(R-r)^{-\alpha}+B\Big).$$ \end{lemma}

\section{Sufficient condition for $\WEHI^+(\phi)$}\label{S:3}

In this section, we will establish the following, which
gives a sufficient condition for $\WEHI^+(\phi)$.

\begin{theorem}\label{P:essuper}
Assume that $\VD$, \eqref{pvd1}, \eqref{pvd}, $\J_{\phi,\le}$, $\FK(\phi)$, $\PI(\phi)$ and  $\CSJ(\phi)$
hold. Then, $\WEHI^+(\phi)$ holds. More precisely,
there exist
constants $\varepsilon\in (0,1)$ and $c\ge1$ such that for all $x_0\in M$,
$0<r<R/ (60 \kappa)$ and any bounded  $\sE$-superharmonic function $u$ on $B_R:=B(x_0,R)$ with $u\ge 0$ on $B_R$,
\[
\left(\frac{1}{\mu(B_r)}\int_{B_r} u^\varepsilon
\,d\mu\right)^{1/\varepsilon}\le
c\Big(\essinf_{B_r}u+\phi(x_0,r){\T_\phi}\, (u_-;
x_0,R)\Big),
\]where $\kappa\ge 1$ is the constant in $\PI(\phi)$ and $B_r=B(x_0,r)$.

\end{theorem}

Throughout this section, we always assume that $\mu$ and $\phi$ satisfy $\VD$, \eqref{pvd1} and \eqref{pvd}, respectively. To prove Theorem \ref{P:essuper} we mainly follow \cite{CKP2}, which
is originally due to \cite{DT}. Since we essentially make use of
$\CSJ(\phi)$, some nontrivial modifications are required. We begin with the following
result, which easily follows from \cite[Corollary 4.12]{CKW2}.

\begin{lemma}\label{C:PIandlog} Let $B_r=B(x_0,r)$ for some $x_0\in M$ and $r>0$.
Assume that $u$ is a bounded and $\sE$-superharmonic function on
$B_R$ such that $u\ge 0$ on $B_R$.  For any $a,l>0$ and $b>1$,
define
$$v= \Big[\log \Big(\frac{a+l}{u+l}\Big)\Big]_+ \wedge \log b.$$ If $\VD$, \eqref{pvd1}, \eqref{pvd},  $\J_{\phi, \le}$, $\PI(\phi)$ and $\CSJ(\phi)$ hold, then for
any $l>0$ and
$0<r\le R/(2 \kappa)$,
$$\frac{1}{V(x_0,r)}\int_{B_r} (v-\ol v_{B_r})^2 d\mu\le  c_1\bigg(1+  \frac{\phi(x_0,r){\T_\phi}\,(u_-; x_0,R)}{l}\bigg),$$
 where $\kappa\ge 1$ is the constant in $\PI(\phi)$, $\ol v_{B_r}= \frac{1}{\mu(B_r)}\int_{B_r}v\,d\mu$ and $c_1$ is a constant independent of $u$, $x_0$, $r$, $R$ and $l$.
\end{lemma}

\begin{lemma}\label{L:EHIone} Let $B_r=B(x_0,r)$ for some $x_0\in M$ and $r>0$.
Assume that $u$ is a bounded and  $\sE$-superharmonic function on $B_R$ such that $u\ge 0$ on $B_R$. Assume that $\VD$, \eqref{pvd1}, \eqref{pvd},  $\J_{\phi,\le}$, $\FK(\phi)$, $\PI(\phi)$ and $\CSJ(\phi)$
hold.
Suppose that there exist constants $\lambda>0$ and $\sigma\in (0,1]$ such that
\begin{equation}\label{ee:cond1}\mu(B_r\cap \{u\ge \lambda\})\ge \sigma\mu(B_r)\end{equation} for
some $r$ with
$0< r<R/(12\kappa)$,
where $\kappa\ge 1$ is the constant in
$\PI(\phi)$. Then there exists a constant $c_1>0$ such that
$$\frac{\mu(B_{6r}\cap \{u\le 2\delta \lambda -\frac{1}{2}{\phi(x_0,r)}{\T_\phi}\, (u_-; x_0,R)\})}{\mu(B_{6r})}\le \frac{c_1}{\sigma \log\frac{1}{2\delta}}$$
holds for all $\delta\in(0,1/4)$, where $c_1$ is a constant
independent of $u$, $x_0$, $r$, $R$, $\sigma$, $\lambda$ and
$\delta$.
\end{lemma}
\begin{proof}Taking $l=\frac{1}{2}{\phi(x_0,r)}{\T_\phi}\, (u_-; x_0,R),$ $a=\lambda$ with $\lambda>0$,
and $b=\frac{1}{2\delta}$ with $\delta\in (0,1/4)$ in Lemma
\ref{C:PIandlog}, we get that for all $\lambda>0$ and
$0<r<R/(12\kappa)$,
\begin{equation}\label{e:pie}\frac{1}{V(x_0,6r)}\int_{B_{6r}}|v-\ol v_{B_{6r}}|\,d\mu\le
\left(\frac{1}{V(x_0,6r)}\int_{B_{6r}}
(v-\ol v_{B_{6r}})^2\,d\mu\right)^{1/2}\le c_1,\end{equation} where
$$v=\min\left\{\left[\log\Big(\frac{\lambda+l}{u+l}\Big)\right]_+,\,\,\log\frac{1}{2\delta}\right\}.$$

Notice that by the definition of $v$, we have $\{v=0\}=\{u\ge
\lambda\}.$ Hence, by \eqref{ee:cond1} and $\VD$, for some $r$ with
$0<r<R/(12\kappa)$,
$$\mu(B_{6r}\cap\{v=0\})\ge c'\sigma \mu(B_{6r})$$ and so
\begin{align*}\log\frac{1}{2\delta}=&\frac{1}{\mu(B_{6r}\cap\{v=0\})}\int_{B_{6r}\cap\{v=0\}}\left(\log\frac{1}{2\delta}-v\right)\,d\mu\\
\le&\frac{1}{c'\sigma} \frac{1}{\mu(B_{6r})}\int_{B_{6r}}\left(\log\frac{1}{2\delta}-v\right)\,d\mu\\
\le&\frac{1}{c'\sigma} \left(\log\frac{1}{2\delta}-\ol v_{B_{6r}}\right).\end{align*}Thus, integrating the previous inequality over
$B_{6r}\cap\{v=\log \frac{1}{2\delta}\}$, we obtain
\begin{align*}\Big(\log\frac{1}{2\delta}\Big)\mu\left(B_{6r}\cap\Big\{v=\log \frac{1}{2\delta}\Big\}\right)\le&\frac{1}{c'\sigma} \int_{B_{6r}\cap\{v=\log \frac{1}{2\delta}\}} \left(\log\frac{1}{2\sigma}-\ol v_{B_{6r}}\right)\,d\mu\\
\le& \frac{1}{c'\sigma} \int_{B_{6r}}|v-\ol v_{B_{6r}}|\,d\mu\\
\le& \frac{c_2}{\sigma} V(x_0,6r),\end{align*} where in the last
inequality we have used \eqref{e:pie}. Therefore, for all
$0<\delta<1/4$,
$$\mu(B_{6r}\cap \{u+l\le 2\delta(\lambda+l)\})\le \frac{c_2}{\sigma \log \frac{1}{2\delta} } V(x_0,6r),$$ which proves the desired assertion.  \qed \end{proof}

For any $x\in M$ and $r>0$, set
$B_r(x)=B(x,r)$. For a ball $B \subset M $
and a function $w$ on $B$, write
$$ I(w,B) = \int_B w^2 \, d\mu. $$
The following lemma can be proved similarly as the of \cite[Lemma 4.8]{CKW1}.

\begin{lemma} \label{oppo}
Suppose $\VD$, \eqref{pvd1}, \eqref{pvd}, $\J_{\phi,\le}$, $\FK(\phi)$ and
$\CSJ(\phi)$ hold. For $x_0 \in M$, $R, {r_1},{r_2}>0$  with
 ${r_1}\in [\frac12 R,R]$ and ${r_1}+{r_2}\le R$, let $u$ be an $\sE$-subharmonic
function on $B_R(x_0)$, and $v =
(u-\theta)_+$ for some $\theta>0$. Set $ I_0 = I(u,B_{{r_1}+{r_2}}(x_0))$ and $ I_1
= I(v,B_{{r_1}}(x_0))$. We have
$$
 I_1\le
\frac{ c_1 }{ \theta^{2\nu} V(x_0,R)^\nu} I_0^{1+\nu}
\left(1+\frac{{r_1}}{{r_2}}\right)^{\beta_2} \left[ 1+
\left(1+\frac{{r_1}}{{r_2}}\right)^{d_2+\beta_2-\beta_1}
\frac{\phi(x_0,R){\T_\phi}\,(u; x_0,R/2)}{\theta} \right],
$$
where $\nu$ is the constant in
$\FK(\phi)$, $d_2$ is the constant in
\eqref{e:vd2}, $\beta_1$ and $\beta_2$ are the constants in \eqref{pvd1}, and $c_1$ is a constant independent of $\theta, x_0, R, {r_1}$,
${r_2}$ and $u$.
\end{lemma}

We also need the following elementary iteration lemma, see, e.g.,
\cite[Lemma 7.1]{Giu} or \cite[Lemma 4.9]{CKW1}.

\begin{lemma}\label{L:it} Let $\beta>0$ and let $\{A_j\}$ be a sequence of real positive numbers such that
$$A_{j+1}\le c_0b^jA_j^{1+\beta},\quad j\ge 1$$ with $c_0>0$ and $b>1$. If $$A_0\le c_0^{-1/\beta}b^{-1/\beta^2},$$ then we have $$A_j\le b^{-j/\beta}A_0,\quad j\ge0,$$ which in particular yields $\lim_{j\to\infty}{ A_j}=0.$ \end{lemma}

The following proposition gives us the infimum of the superharmonic function. This extends the analogous expansion of positivity in the local setting, which is a key step towards $\WEHI^+(\phi)$.

\begin{proposition}\label{P:infsuph} Let $B_r=B(x_0,r)$ for some $x_0\in M$ and any $r>0$.
Assume that $u$ is a bounded and $\sE$-superharmonic function on $B_R$ such that $u\ge 0$ on $B_R$. Assume that $\VD$, \eqref{pvd1}, \eqref{pvd}, $\J_{\phi,\le}$, $\FK(\phi)$, $\PI(\phi)$ and  $\CSJ(\phi)$
hold.  Suppose that there exist constants $\lambda>0$ and $\sigma\in(0,1]$ such that
\be\label{eq:WHI-GHH-as}
\mu(B_r\cap \{u\ge \lambda\})\ge \sigma\mu(B_r)
\ee
for some $r$ satisfying
 $ 0<r<R/(12\kappa )$, where $\kappa\ge 1$ is the constant in $\PI(\phi)$.  Then, there exists a constant $\delta\in(0,1/4)$ depending on $\sigma$ but independent of
   $\lambda$, $r, R$,   $x_0$ and $u$,  such that
\be\label{eq:WHI-GHH}
\essinf_{B_{4r}}u\ge \delta \lambda -\phi(x_0,r){\T_\phi}\, (u_-; x_0,R).\ee
\end{proposition}
\begin{proof} Without loss of generality, we may and do assume that
\begin{equation}\label{e:ehise-01} \phi(x_0,r){\T_\phi}\,(u_-; x_0,R)\le \delta \lambda;\end{equation} otherwise the conclusion is trivial due to the fact that $u\ge0$ on $B_R$.

For any $j\ge0$, define
$$l_j=\delta \lambda+2^{-j-1}\delta \lambda,\quad r_j=4r+2^{1-j}r.$$ Then, by \eqref{e:ehise-01} we see that
$$l_0= \frac{3}{2} \delta \lambda\le 2\delta \lambda- \frac{1}{2}{\phi(x_0,r)}{\T_\phi}\,(u_-; x_0,R).$$ Lemma \ref{L:EHIone} implies that
\begin{equation}\label{e:ehise-02}\frac{\mu(B_{6r}\cap \{u\le l_0\})}{\mu(B_{6r})}\le\frac{c_1}{\sigma \log \frac{1}{2\delta} }.\end{equation}
In the following, let us denote by $$B_j=B_{r_j},\quad
w_j=(l_j-u)_+,\quad  A_j=\frac{\mu(B_j\cap \{u\le l_j\})}{\mu(B_j)}.
$$ Note that, $-u$ is an $\sE$-subharmonic function on $B_R$. Then, we have by Lemma \ref{oppo} that
\begin{align*} &A_{j+2}(l_{j+1}-l_{j+2})^2\\
&=\frac{1}{\mu(B_{j+2})}\int_{B_{j+2}\cap \{u\le l_{j+2}\}} (l_{j+1}-l_{j+2})^2\,d\mu\\
&\le \frac{1}{\mu(B_{j+2})}\int_{B_{j+2}} w_{j+1}^2\,d\mu\\
&\le \frac{c_2}{(l_j-l_{j+1})^{2\nu}}\left( \frac{1}{\mu(B_{j+1})}\int_{B_{j+1}} w_{j}^2\,d\mu\right)^{1+\nu}\left(\frac{r_{j+2}}{r_{j+1}-r_{j+2}}\right)^{\beta_2}\\
&\quad \times \left[1+\frac{1}{l_j-l_{j+1}}\left(\frac{r_{j+2}}{r_{j+1}-r_{j+2}}\right)^{d_2+
\beta_2-\beta_1}\phi(x_0,r_{j+1}) {\T_\phi}\,(w_j; x_0, r_{j+1})\right]\\
&\le \frac{c_3}{[(2^{-j}-2^{-j-1})\delta \lambda]^{2\nu}}\left[(\delta \lambda)^2 A_j\right]^{1+\nu}\left(\frac{1}{2^{-j}-2^{-j-1}}\right)^{\beta_2}\\
&\quad \times \left[1+\frac{1}{(2^{-j}-2^{-j-1})\delta
\lambda}\left(\frac{1}{2^{-j}-2^{-j-1}}\right)^{d_2+
\beta_2-\beta_1} \phi(x_0, r_{j+1}){\T_\phi}\,(w_j; x_0, r_{j+1})\right]\\
&\le c_{4} (\delta \lambda)^2 A_j^{1+\nu}
2^{(1+2\nu+d_2+2\beta_2-\beta_1)j}\left(1+\frac{ 1}{\delta
\lambda}\phi(x_0,r_{j+1}){\T_\phi}\,(w_j; x_0, r_{j+1})\right),
\end{align*} where $\nu$ is the constant in $\FK(\phi)$, and in the third inequality we have used the facts that $w_j\le l_j\le 3\delta \lambda/2$ and
$$\int_{B_{j+1}}w_j^2\,d\mu\le c'(\delta \lambda)^2 \mu(B_{j+1}\cap \{w_j\ge0\})\le c''(\delta \lambda )^2 \mu(B_j\cap\{u \le l_j\}).$$

Note that \begin{align*}&\phi(x_0,r_{j+1}){\T_\phi}\, (w_j; x_0, r_{j+1})\\
&=\phi(x_0,r_{j+1})\int_{B_{j+1}^c}\frac{|w_j|(z)}{V(x_0,d(x_0,z))\phi(x_0,d(x_0,z))}\,\mu(dz)\\
&\le \phi(x_0,r_{j+1}) \int_{B_{j+1}^c}\frac{l_j+u_-(z)}{V(x_0,d(x_0,z))\phi(x_0,d(x_0,z))}\,\mu(dz)\\
&=\phi(x_0,r_{j+1})\bigg[\int_{B_R\setminus B_{j+1}} \frac{l_j}{V(x_0,d(x_0,z))\phi(x_0,d(x_0,z))}\,\mu(dz)\\
&\qquad\qquad\quad +\int_{B_R^c} \frac{l_j+u_-(z)}{V(x_0,d(x_0,z))\phi(x_0,d(x_0,z))}\,\mu(dz)\bigg]\\
&\le c_5\bigg(l_j+ \frac{\phi(x_0,r_{j+1})}{\phi(x_0,R)}l_j+ {\phi(x_0,r)}{\T_\phi}\,(u_-;x_0,R)\bigg)\\
&\le c_6 \delta \lambda,
\end{align*} where in the second equality we have used the fact that $u\ge0 $ on $B_R$, in the second inequality we used  Remark \ref{intelem}, and the last inequality follows from \eqref{e:ehise-01}.

According to all the estimates above, we see that there is a constant $c_7>0$ such that for all $j\ge0$,
$$A_{j+2}\le c_7 A_j^{1+\nu} 2^{(3+2\nu+d_2+2\beta_2-\beta_1)j}.$$
Let $c^*=c_{7}^{-1/\nu}2^{-(3+2\nu+d_2+2\beta_2-\beta_1)/\nu^2}$ and choose the constant $\delta\in\big(0,1/4\big)$ such that
$$\frac{c_1}{\sigma \log \frac{1}{2\delta} }\le c^*,$$ then, by \eqref{e:ehise-02},  $$A_0\le c^*= c_{7}^{-1/\nu}2^{-(3+2\nu+d_2+2\beta_2-\beta_1)/\nu^2}.$$ According to Lemma \ref{L:it}, we can deduce that $\lim_{i\to\infty}A_i=0$. Therefore, $u\ge \delta \lambda$ on $B_{4r}$, from which the desired assertion follows easily. \qed \end{proof}

\begin{remark}\rm
Proposition \ref{P:infsuph} contains \cite[Lemma 4.5]{GHH} as a special case. Among other things,
 in \cite[Lemma 4.5]{GHH},
it is assumed that $\phi (x, r)=r^\alpha$ for $\J_{\phi}$. Inequality \eqref{eq:WHI-GHH} under condition
 \eqref{eq:WHI-GHH-as}
is called a weak Harnack inequality in \cite{GHH}.
\end{remark}

Below is
a Krylov-Safonov covering lemma on metric measure spaces, whose
proof is essentially taken from \cite[Lemma 7.2]{KS}.
Note that the difference of the following definition of $[E]_\eta$
from that in
\cite[Lemma 7.2]{KS} is that here we
impose the restriction of $0<\rho <r$
and change the constant $3$ to
$5$. For the sake of completeness, we present the proof here.

\begin{lemma}\label{L:KS} Suppose that $\VD$ holds. Let $B_r(x_0)=B(x_0,r)$ for some $x_0\in M$ and any $r>0$. Let $E\subset B_r(x_0)$ be a measurable set. For any $\eta\in (0,1)$, define
$$[E]_\eta=\bigcup_{0<\rho<r}\bigg\{B_{5\rho}(x)\cap B_r(x_0): x\in B_r(x_0)\textrm{ and }\frac{\mu(E\cap B_{5\rho}(x))}{\mu(B_{\rho}(x))}>\eta\bigg\}.$$ Then, either $$[E]_\eta=B_r(x_0)$$ or $$\mu([E]_\eta)\ge \frac{1}{\eta}\mu(E).$$ \end{lemma}

\begin{proof}
Define a
maximal operator $A:B_r(x_0)\to [0,\infty)$ as follows
$$A(x)=
\sup_{y\in B_r(x_0), x\in B_{5\rho}(y), 0<\rho<r}
\frac{\mu(E\cap B_{5\rho}(y))}{\mu(B_{\rho}(y))} .$$ We claim that
$$[E]_\eta=\{x\in B_r(x_0): A(x)>\eta\}$$ for any $\eta\in(0,1)$. Indeed, let $x\in B_r(x_0)$ with $A(x)>\eta$. Then, there is a ball $B_\rho(y)$ with $0<\rho<r$, $y\in B_r(x_0)$ and $x\in B_{5\rho}(y)$ such that $\frac{\mu(E\cap B_{5\rho}(y))}{\mu(B_{\rho}(y))}>\eta$. This means that $$x\in\bigg\{B_{5\rho}(y)\cap B_r(x_0): y\in B_r(x_0)\textrm{ and }\frac{\mu(E\cap B_{5\rho}(y))}{\mu(B_{\rho}(y))}>\eta\bigg\}\subset [E]_\eta.$$ On the other hand, if $x\in [E]_\eta$, then there is a ball $B_\rho(y)$ with  $0<\rho<r$, $y\in B_r(x_0)$ and $x\in B_{5\rho}(y)$ such that $\frac{\mu(E\cap B_{5\rho}(y))}{\mu(B_{\rho}(y))}>\eta$. This implies that $A(x)>\eta.$

Suppose that $B_r(x_0)\setminus [E]_\eta\ne \emptyset$.
The set $[E]_\eta$ is open by definition. We cover $[E]_\eta$ by ball
$B_{r_x}(x)$, where $x\in [E]_\eta$ and $r_x=d(x, B_r(x_0)\setminus
[E]_\eta)/2\in(0,r)$.
By the Vitali covering lemma,{\footnote{The $\VD$ condition implies a covering theorem, here referred
to as the Vitali covering theorem. Indeed, given any collection of balls
with uniformly bounded radius, there exists a pairwise disjoint, countable
subcollection of balls, whose $5$-dilates cover the union of the original
collection. See \cite[Chapter 1]{H} for more details.}} there are
countably many pairwise disjoint balls $B_{r_i}(x_i)$, where
$r_i=r_{x_i}$ for all $i\ge1$, such that $[E]_\eta\subset
\bigcup_{i=1}^\infty B_{5r_i}(x_i).$ Note that, $B_{5r_i}(x_i)\cap
(B_r(x_0)\setminus [E]_\eta)\neq \emptyset$
for all $i\ge1$, and so
there is a point $y_i\in B_{5r_i}(x_i)\cap  (B_r(x_0)\setminus
[E]_\eta)$. In particular, $A(y_i)\le \eta$ for all $i\ge1$. Since $y_i\in B_r(x_0)$,
$x_i\in B_{5r_i}(y_i)$ and $0<r_i<r$, we conclude that
$$\mu (E\cap B_{5r_i}(x_i))\le\eta\mu(B_{r_i}(x_i)).$$ If $y$ is the density point $E$ (i.e. $y\in E$ such that $\lim_{\rho\to0} \frac{\mu(E\cap B_{\rho}(y))}{\mu(B_{\rho}(y))}=1)$, then
$$
\liminf_{\rho\to0} \frac{\mu(E\cap B_{5\rho}(y))}{\mu(B_{\rho}(y))}\ge \lim_{\rho\to0} \frac{\mu(E\cap B_{\rho}(y))}{\mu(B_{\rho}(y))}=1>\eta.
$$

Since $\mu$-almost every point $E$ is a density point; that is, for almost all $x\in E$, it holds that $\lim_{\rho\to0} \frac{\mu(E\cap B_\rho(x))}{\mu(B_\rho(x))}=1$, which follows from $\VD$ and the Lebesgue differentiation theorem (see \cite[Theorem 1.8]{H}), we observe that $\mu$-almost every point of $E$ belongs to $[E]_\eta$ for every $\eta\in (0,1)$. From this it follows that
$$\mu(E)= \mu (E\cap [E]_\eta)\le \sum_{i=1}^\infty \mu (E \cap B_{5r_i}(x_i))\le \eta \sum_{i=1}^\infty \mu(B_{r_i}(x_i)) \le \eta \mu([E]_\eta).$$ The above inequality yields the desired result.  \qed \end{proof}

We now are in a position to present the

\medskip

\noindent{\bf Proof of Theorem \ref{P:essuper}.}\quad Fix $\eta\in
(0,1)$. Let us define for any $t>0$ and $i\ge 0$
$$A_t^i=\Big\{ x\in B_r(x_0): u(x)>t\delta^i-\frac{T}{1-\delta}\Big\},$$ where $\delta\in(0,1)$ is
determined later and $$T={\phi(x_0,5r)}{\T_\phi}\, (u_-; x_0, R).$$ Obviously we have $A_t^{i-1}\subset A_t^i$ for all $i\ge 1$.
If there are a point $x\in B_r(x_0)$ and $0<\rho<r$ such that $B_{5\rho}(x)\cap B_r(x_0) \subset [A_t^{i-1}]_\eta$,
then, by the definition of $[A_t^{i-1}]_\eta$ and $\VD$,
$$\mu(A_t^{i-1}\cap B_{5\rho}(x))\ge \eta \mu(B_\rho(x))\ge c'\eta \mu(B_{5\rho}(x)),$$ where $c'$ is
a positive constant independent of $\eta$, $x_0$ and $\rho$. Below,
let $\delta$ be the constant given in Proposition \ref{P:infsuph}
corresponding to the factor $c'\eta$. Applying Proposition
\ref{P:infsuph}  with $\lambda= t\delta^{i-1}-\frac{T}{1-\delta}$
and $\sigma =c'\eta$, we get that
\begin{align*}\essinf_{B_{20\rho}(x)}u> &\delta
\left(t\delta^{i-1}-\frac{T}{1-\delta}\right)-
{\phi(x_0,5\rho)}{\T_\phi}\, (u_-; x_0, R)\\
 \ge &\delta
\left(t\delta^{i-1}-\frac{T}{1-\delta}\right)-T\\
=& t\delta^{i}-\frac{T}{1-\delta}.\end{align*} Hence, if
$B_{5\rho}(x)$ is one of the balls to make up to the set
$[A_t^{i-1}]_\eta$ in Lemma \ref{L:KS}, then $B_{5\rho}(x)\cap
B_r(x_0)\subset A_t^i$, which implies that $[A_t^{i-1}]_\eta\subset
A_t^i$. By Lemma \ref{L:KS}, we must have either $A_t^i=B_r(x_0)$ (since $A_t^i\subset B_r(x_0)$ for all $i\ge0$) or
\begin{equation}\label{e:econ01}\mu(A_t^i)\ge \frac{1}{\eta}\mu(A_t^{i-1}).\end{equation}

We choose an integer $j\ge 1$ so that
$$\eta^j< \mu(A_t^0)/\mu(B_r(x_0))\le \eta^{j-1}.$$ Suppose first that $A_t^{j-1}  \neq B_r(x_0)$. Then, by the fact that $A_t^{i-1}\subset A_t^i$ for all $i\ge 1$, we have
$A_t^{k}\neq B_r(x_0)$ for all $0\le k\le j-1$. Hence, according to
\eqref{e:econ01}, we obtain that
$$\mu(A_t^{j-1})\ge \frac{1}{\eta}\mu(A_t^{j-2})\ge \cdots\ge \frac{1}{\eta^{j-1}} \mu(A_t^0)\ge \eta \mu(B_r(x_0)).$$ Note that, the inequality holds trivially for the case that
$A_t^{j-1}=B_r(x_0)$, thanks to the fact that $\eta\in(0,1)$.
Therefore, according to Proposition \ref{P:infsuph} again, we have
\begin{align*}\essinf_{B_{4r}(x_0)}u\ge &c_1\left(t\delta^{j-1}-\frac{T}{1-\delta}\right)-T\\
\ge& c_1t\delta^{j-1}-\frac{c_2T}{1-\delta}\\
\ge & c_1t\left(\frac{ \mu(A_t^0)}{\mu(B_r(x_0))
}\right)^{1/\gamma}- \frac{c_2T}{1-\delta},\end{align*} where $c_1$
is the constant given in Proposition \ref{P:infsuph} corresponding
to the factor $\delta$, and $\gamma=\log_\delta \eta.$  This is,
$$ \frac{ \mu(A_t^0)}{\mu(B_r(x_0)) } \le \frac{c_3}{t^\gamma}\left(\essinf _{B_{4r}(x_0)} u+ \frac{T}{1-\delta}\right)^\gamma.$$

By Cavalieri's principle, we have for any $0<\eps<\gamma$ and $a>0$,
\begin{align*}\frac{1}{\mu(B_r(x_0))}\int_{B_r(x_0)} u^\epsilon\,d\mu&=\eps \int_0^\infty t^{\eps-1} \frac{\mu(B_r(x_0)\cap \{u>t\})}{\mu(B_r(x_0))}\,dt\\
&\le \eps \int_0^\infty t^{\eps-1} \frac{\mu(A_t^0)}{\mu(B_r(x_0))}\,dt\\
&\le \eps\left[\int_0^a t^{\eps-1}\,dt+ c_3\left(\essinf _{B_{4r}(x_0)}
u+ \frac{T}{1-\delta}\right)^\gamma\int_a^\infty
t^{\eps-1-\gamma}\,dt\right]\\
&\le c_4\bigg[ a^{\eps}+ \left(\essinf _{B_{4r}(x_0)}
u+ \frac{T}{1-\delta}\right)^\gamma a^{\eps-\gamma}\bigg].
\end{align*} In particular, taking $$a=
\essinf _{B_{4r}(x_0)} u+ \frac{T}{1-\delta},$$ we finally get
that
$$\frac{1}{\mu(B_r(x_0))}\int_{B_r(x_0)} u^\epsilon\,d\mu \le c_4 \left(\essinf _{B_{4r}(x_0)} u+ \frac{T}{1-\delta}\right)^\eps.$$ This along with \eqref{pvd1} concludes the proof. \qed

\begin{remark}\rm $\WEHI^+(\phi)$ is equivalent to
the inequality \eqref{eq:WHI-GHH} under $\VD$ and condition \eqref{eq:WHI-GHH-as}.
Indeed, the proof above shows that \eqref{eq:WHI-GHH} under $\VD$ and condition \eqref{eq:WHI-GHH-as} implies $\WEHI^+(\phi)$.
Conversely, assume $\WEHI^+ (\phi)$ holds. Under
condition \eqref{eq:WHI-GHH-as} we have

\[
\left(\frac{1}{\mu(B_r)}\int_{B_r} u^\varepsilon\,d\mu\right)^{1/\varepsilon}
\ge \left(\frac{1}{\mu(B_r)}\int_{B_r\cap \{u\ge \lambda\}} u^\varepsilon\,d\mu\right)^{1/\varepsilon}
\ge \lambda \sigma^{1/\eps}.
\]
Plugging this into $\WEHI^+(\phi)$ yields
\eqref{eq:WHI-GHH}.
\end{remark}

\section{Implications of $\EHI$}
\label{S:4}

In this section, we first study the relation between
$\EHI$ and  $\EHI (\phi)$, and then show that under some conditions
$\EHI$ implies $\PI (\phi)$.

\subsection{$\EHI+\E_{\phi,\le}+\J_{\phi,\le} \Longrightarrow\EHI(\phi)+\FK(\phi)$}\label{S:4.1}

We first recall the following L\'{e}vy system formula. See, for example \cite[Appendix A]{CK2} for a proof.

\begin{lemma}\label{Levy-sys}
Let $f$ be a non-negative measurable function on ${\mathbb R}_+
\times M \times M$ that vanishes along the diagonal. Then for every
$t\geq 0 $, $x\in M_0$ and  stopping time $T$ $($with
respect to the filtration of $\{X_t\}$$)$,
$$
{\mathbb E}^x \left[\sum_{s\le T} f(s,X_{s-}, X_s) \right]={\mathbb
E}^x \left[ \int_0^T  \int_M f(s,X_s, y)\, J(X_s,dy)
\,ds \right].
 $$
\end{lemma}

For any open subset $D\subset M$, denote the transition semigroup
of the part process $X^D=\{X^D_t\}$ of $X$ killed upon leaving $D$ by $\{P^D_t\}$; that is, for any $f\in {\cal B}_+ (D)$,
$$
P^D_t f (x) =\bE^x \left[ f(X^D_t) \right]=\bE^x [ f(X_t); t< \tau_D], \quad x\in D\cap M_0.
$$
The Green operator $G_D$
is defined by
$$
G_D f(x) = \int_0^\infty P^D_t  f (x)\, dt = \bE^x \left[ \int_0^{\tau_D}  f(X_t) \,dt \right],
\quad x\in D\cap M_0.
$$
For $f\in {\cal B} (D)$, if $P^D_t |f|(x) < \infty$, we define
$$
P^D_t f (x) = P^D_t f ^+(x) - P^D_t f^-(x)= \bE^x \left[ f(X^D_t)\right];
$$
if $G_D |f| (x)<\infty$, we define
$$
 G_{D} f(x)=G_D f^+ (x)-G_D f^- (x)   =\bE^x\left[ \int_0^{\tau_{D}} f(X_t)\,dt\right].
$$
It is known that $\{P^D_t\}$ is the semigroup associated with the part Dirichlet form $(\sE, \sF_D)$
of $(\sE, \sF)$ on $D$, where $\sF_D$ is the $\sqrt{\sE_1}$-completion of $\sF\cap C_c (D)$.
  Let
  $\sL_{D}$ be the $L^2$-generator of $(\sE, \sF_{D})$ on $L^2(D; \mu)$. The principle eigenvalue $\lambda_1 (D)$
of $-\sL_{D}$ is defined to be
$$
\lambda_1({D}): = \inf \left\{ \sE(f, f):  \ f\in \sF_D \hbox{ with } \| f \|_{L^2(D; \mu)} =1 \right\}.
$$

The following two lemmas are known, see \cite[Lemma 3.2]{GT} and \cite[Lemma 5.1]{GH2} respectively.

\begin{lemma}\label{green1}
If $\sup_{y\in D\setminus {\cal N}} \bE^y \tau_D<\infty$, then $G_{D}$ is a bounded operator on ${\cal B}_b({D})$, and it uniquely extends to the space $L^p({D};\mu)$ with $p=1,2,\infty$ and  enjoys the following norm estimate
$$\|G_{D}\|_{L^p(D;\mu)\to L^p (D;\mu)}\le \sup_{y\in D\setminus {\cal N}} \bE^y \tau_D .
$$
 Moreover, $G_{D}$ is the reverse of the operator $-\sL_{D}$ in $L^2({D};\mu)$, and
$$ \lambda_1({D})^{-1}\le \sup_{y\in D\setminus {\cal N}} \bE^y \tau_D .
$$
\end{lemma}

\begin{lemma}\label{gh} Let $(\sE,\sF)$ be a regular Dirichlet form in $L^2(M; \mu)$, and  ${D} \subset M$  an open set such that
$\lambda_1({D})>0$. Then the  following hold:
\begin{itemize}
\item[{\rm (i)}] For any $f\in L^2({D; \mu})$,
$$\|G_{D} f\|_{2}\le \lambda_1({D})^{-1} \|f\|_{2}.$$
\item[{\rm (ii)}] For any $f\in L^2({D}; \mu )$,  $P^D_t f \in \sF_D$ for every $t>0$ and  $G_{D} f\in \sF_{D}$. Moreover,
$$\sE (G_{D} f, g)= \langle f,g\rangle  \quad \hbox{for every } g\in \sF_{D}.
$$
\end{itemize}
\end{lemma}

 \begin{theorem}\label{P:ehii} Assume that $\VD$, \eqref{pvd1},
\eqref{pvd}, $\E_{\phi,\le}$
 and $\J_{\phi,\le}$ hold. Then $\EHI$ implies $\EHI(\phi)$. \end{theorem}
\begin{proof}
Fix $x_0\in M$, and assume that  $u$ is harmonic on $B_R:=B(x_0,R)$ such that $u\ge0 $ on $B_R$. Let $\delta\in(0,1)$ be the constant in $\EHI$. Then, for almost all $x,y\in  B_{\delta r}$  with $0<r\le \delta R$,
\begin{align*}
u(x)&=\bE^xu(X_{\tau_{B_r}})\le \bE^xu_+(X_{\tau_{B_r}})\le c\bE^yu_+(X_{\tau_{B_r}})\\
&=c\left(\bE^yu(X_{\tau_{B_r}})+\bE^yu_-(X_{\tau_{B_r}})\right)\\
&= c u(y)+c
\, \bE^yu_-(X_{\tau_{B_r}}),
\end{align*}
where in the second inequality we used $\EHI$.
Since $u_-=0$ on $B_R$,  by the L\'evy system of $X$ and condition $\J_{\phi,\le}$,
we have for $y\in B_{\delta r}$
\begin{align*}
 \bE^yu_-(X_{\tau_{B_r}})&=\bE^y\left[ \int_0^{\tau_{B_r}}\int_{B_R^c} u_-(z)J(X_s, z)\,\mu(dz)\,ds\right]\\
&=\int_{B_R^c} u_- (z)  \left( G_{B_r} J(\cdot,z) \right) (y)  \,\mu (dz) \\
&\le  c_1 \int_{B_R^c} u_- (z)
\left( G_{B_r} \frac{1}{V(\cdot,d(\cdot,z))\phi(\cdot, d(\cdot,z))}\right) (y)
\,\mu (dz)\\
&\le c_2\int_{B_R^c}  \frac{ u_- (z) }{V(x_0,d(x_0,z))\phi(x_0, d(x_0,z))} \,\mu (dz)
\, G_{B_r}{\bf 1}(y)  \\
&=  c_2\int_{B_R^c}  \frac{u_- (z) }{V(x_0,d(x_0,z))\phi(x_0, d(x_0,z))} \,  \mu (dz)
\, \bE^y\tau_{B_r} \\
& \le c_3\phi(x_0,r)  \int_{B_R^c}  \frac{u_- (z) }{V(x_0,d(x_0,z))\phi(x_0, d(x_0,z))}\, \mu (dz)\\
&= c_3 \phi(x_0,r) \, {\T_\phi} (u_-; x_0, R),
\end{align*}
where the second inequality follows from $\VD$, \eqref{pvd1} and \eqref{pvd}, and in the last inequality we have used $\E_{\phi,\le}$ and \eqref{pvd}.
\qed
\end{proof}

\medskip

From Theorem \ref{P:ehii}, we can deduce the following.

\begin{proposition}\label{P:fk}
{\rm (i)} Assume that $\VD$, \eqref{pvd1} and \eqref{pvd} hold. Then,
$$\EHR+\E_{\phi,\le}\Longrightarrow\FK(\phi).$$

{\rm (ii)} Assume that $\VD$, \eqref{pvd1}, \eqref{pvd}, $\E_{\phi,\le}$ and $\J_{\phi,\le}$ hold. Then $\EHI$ implies $\FK(\phi)$.
\end{proposition}

\begin{proof} The first required assertion follows from the argument of \cite[Lemma 4.6]{CKW2}. By Theorems \ref{P:ehii} and \ref{P:holder}, we have $\EHI+\E_{\phi,\le}+\J_{\phi,\le}\Longrightarrow \EHR$, which along with the first assertion immediately yields the second one.\qed
\end{proof}

\subsection{$\EHI+\E_\phi+\J_{\phi,\le}\Longrightarrow \PI(\phi)$} \label{S:4.2}

\begin{proposition}\label{P:PI}
{\rm (i)} Suppose that $\VD$, \eqref{pvd1}, \eqref{pvd}, $\EHR$, $\E_\phi$ and $\FK(\phi)$ hold. Then $\PI(\phi)$ holds.\\
{\rm (ii)}
Suppose that $\VD$, \eqref{pvd1}, \eqref{pvd}, $\EHI$, $\E_\phi$ and $\J_{\phi,\le}$ hold. Then $\PI(\phi)$ holds.
\end{proposition}

\begin{proof}
By Theorems \ref{P:ehii}, \ref{P:holder} and Proposition \ref{P:fk}, we have $\EHI+\E_{\phi,\le}+\J_{\phi,\le}\Longrightarrow \EHR+\FK(\phi)$, so it is enough to prove (i).
According to the proofs of \cite[Lemmas 5.4 and 5.5]{GH1} and choosing $a=CV(x,r)^\nu/\phi(x,r)$ in that paper, we can get from $\FK(\phi)$ that for any ball $B=B(x,r)$ with $x\in M$ and $r>0$, the Dirichlet heat kernel $p^B(t,x,y)$ exists, and  there exists a constant $C_1>0$ such that
$$\esssup_{y,z\in B}p^B(t,y,z)\le \frac{C_1}{V(x,r)}\left(\frac{\phi(x,r)}{t}\right)^{1/\nu},\quad t>0.$$
Using this estimate, $\J_{\phi,\le}$ and $\EHR$, and following the argument of \cite[Lemma 4.8]{CKW2}, there are constants $\kappa, C_2>0$ such that for any $x\in M_0$, $t>0$, $0<{r}\le 2^{-(\beta_1+\theta)/\beta_1} \phi^{-1}(x,t)$ and $y\in B(x,{r})\backslash \mathcal{N}$,
$$
|p^{B(x,\phi^{-1}(x,t))}(t, x,x)-p^{B(x,\phi^{-1}(x,t))}(t, x,y)|
\le \left( \frac{{r}}{\phi^{-1}(x,t) }\right)^\kappa \frac{C_2}{V(x,\phi^{-1}(x,t))},
$$
where $\beta_1$ is the constant in \eqref{pvd1} and $\theta$ is the H\"older exponent in {\rm EHR}.

Since for any $t>0$, $\tau_{B(x,r)}\le t+(\tau_{B(x,r)}-t){\bf 1}_{\{\tau_{B(x,r)}\ge t\}},$ we have, by the Markov property, $\E_{\phi,\le}$ and \eqref{pvd},
\begin{align*}
 \bE^x\tau_{B(x,r)}&\le t+ \bE^x\Big[{\bf 1}_{\{\tau_{B(x,r)}>t\}} \bE^{X_t}[ \tau_{B(x,r)}-t ]\Big]
 \le t+ \bP^x(\tau_{B(x,r)}>t) \sup_{z\in B(x,r)} \bE^{z}\tau_{B(x,r)}\\
 &\le t+ \bP^x(\tau_{B(x,r)}>t) \sup_{z\in B(x,r)} \bE^{z}\tau_{B(z,2r)}
 \le t+ c_2\bP^x(\tau_{B(x,r)}>t)  \phi(x,r).
\end{align*}
Then, by $\E_\phi$, for all $x\in M_0$, $$c_1\phi(x,r)\le \bE^x\tau_{B(x,r)}\le  t+  c_2\bP^x(\tau_{B(x,r)}>t) \phi(x,r)$$ and so
$$
 \bP^x(\tau_{B(x,r)} \leq t) \le 1- \frac{c_1\phi(x,r)-t}{c_2\phi(x,r)}.$$
In particular, we can choose a constant $\delta>0$ such that for all $r>0$ and all $x\in M_0$,
\begin{equation}\label{e:4.1}
\bP^x \left(\tau_{B(x,r)}\le \delta \phi(x,r) \right)\le
1-\frac{c_1}{2c_2}.
\end{equation}

Combining with all the conclusions above, we can see from the argument of \cite[Proposition 4.9]{CKW2} that there exist $\eps\in (0,1)$ and $c_3>0$ such that for any $x_0\in M$, $r>0$,
$0<t\le \phi(x_0,\eps r)$ and $B=B(x_0,r)$,
$$ p^{B}(t, x
,y )\ge \frac{c_3}{V(x_0, \phi^{-1}(x_0,t))},\quad x ,y\in B(x_0,
\eps\phi^{-1}(x_0,t)) \cap M_0,
$$ which yields $\PI(\phi)$ by some standard arguments, see \cite[Proposition 3.5(i)]{CKW2}.\qed \end{proof}

At the end of this section, we present a consequence of $\E_\phi$ and $\J_{\phi,\le}$ (without $\EHI$).

 \begin{proposition}\label{P:CSJ}
 Under $\VD$, \eqref{pvd1} and \eqref{pvd}, $\E_\phi$ and $\J_{\phi,\le}$ imply $\CSJ(\phi)$. \end{proposition}
\begin{proof}
As shown in \eqref{e:4.1},
under $\VD$, \eqref{pvd1} and \eqref{pvd}, $\E_\phi$ implies that
there are constants $\delta_0 >0$ and $0<\eps_0 <1$ such that for all $r>0$ and all $x\in M_0$,
$$
\bP^x(\tau_{B(x,r)}\le \delta_0 \phi(x,r))\le  \eps_0.
$$
 Having this estimate at hand with $\J_{\phi,\le}$, and following the arguments in \cite[Subsection 3.2]{CKW1} by replacing $\phi(r)$ with $\phi(x_0,r)$, we can prove the desired assertion.\qed
 \end{proof}

 \section{Exit time and relative capacity }\label{S:5}

 In this section, we study the relation between two-sided mean exit time estimates
 and two-sided relative
capacitary estimates.

 \subsection{From mean exit time estimates $\E_\phi$  to relative capacitary estimates}\label{S:5.1}

Recall that for open subsets $A$ and $B$ of $M$ with $A\Subset B$, we define
\[
 \mbox{{\rm Cap}} (A,B)=\inf \left\{\sE(u,u): u\in \sF, \, u=1
 \hbox{ $\sE$-q.e. on } A  \hbox{ and } u=0
  \hbox{ $\sE$-q.e. on } B^c \right\}.
\]
Note that $\mbox{{\rm Cap}} (A,B)$ is increasing
in $A$ but decreasing in $B$.

 \begin{proposition}\label{P:ec}Under $\VD$, \eqref{pvd1} and \eqref{pvd}, if $\E_\phi$ holds, then for any $B_r=B(x_0,r)$ with some $x_0\in M$ and $r>0$,
$$\Ca(B_{r/2}, B_r)\asymp \frac{ V(x_0,r)}{\phi(x_0,r)}.$$  \end{proposition}
\begin{proof}Throughout the proof, define
$g(x)=\bE^x\tau_{B_r}=G_{B_r}{\bf1}$.
Set $$u=\frac{g}{\inf_{B_{r/2}} g} \wedge 1.$$ Then, $u|_{B_{r/2}}=1$, $u|_{B_r^c}=0$ and
$$\sE(u,u)\le \frac{1}{(\inf_{B_{r/2}} g)^2}\sE(g,g)= \frac{\int_{B_r}g\,d\mu}{(\inf_{B_{r/2}}g)^2}\le \frac{c_1V(x_0,r)}{\phi(x_0,r)},$$ where the equality follows from the fact that $\sL_{B_r} g=-1$, and in the second inequality we have used $\E_{\phi}$, \eqref{pvd1} and \eqref{pvd}.
Hence, $$\Ca(B_{r/2},B_r)\le \sE(u,u)\le \frac{c_1V(x_0,r)}{\phi(x_0,r)}.$$

On the other hand, since
$$\frac{\left(\int_{B_r} g\,d\mu\right)^2}{\sE(g,g)} =\frac{\left(\int_{B_r} g\,d\mu\right)^2}{\int_{B_r} g\,d\mu}=\int_{B_r} g\,d\mu,$$ and
for any $u\in \sF_{B_r}$ $$\int_{B_r}u\,d\mu=\sE(u, G_{B_r}{\bf 1})\le \sqrt{\sE(u,u)}\sqrt{\sE(G_{B_r}{\bf 1},G_{B_r}{\bf 1} )}
=\sqrt{\sE(u,u)}\sqrt{\int_{B_r} g\,d\mu},$$ we have
$$\int_{B_r} g\,d\mu=\sup\left\{ \frac{\left(\int_{B_r} u\,d\mu\right)^2}{\sE(u,u)}: u\in \sF_{B_r} \right\}.$$
Applying the inequality above with $\E_\phi$, $\VD$, \eqref{pvd1} and \eqref{pvd}, we find that
\begin{align*}c_2 V(x_0,r) \phi(x_0,r)\ge&\int_{B_r} g\,d\mu =\sup\left\{ \frac{\left(\int_{B_r} u\,d\mu\right)^2}{\sE(u,u)}: u\in \sF_{B_r} \right\}\\
\ge& V(x_0,r/2)^2\sup\left\{ \frac{1}{\sE(u,u)}: u\in \sF_{B_r}, u|_{B_{r/2}}=1 \right\}\\
=&  \frac{c_3V(x_0,r)^2}{\Ca(B_{r/2}, B_r)}. \end{align*} Hence,
$$\Ca(B_{r/2}, B_r)\ge\frac{ c_4 V(x_0,r)}{\phi(x_0,r)}.$$
The proof is complete. \qed\end{proof}

\subsection{From  capacitary estimates to mean exit time estimates
$\E_\phi$}  \label{S:5.2}

In this subsection, we assume that Assumptions  \ref{A1}
and $\ref{A1-b}$ hold,  and will prove that
$$
\EHI+ \J_{\Ex,\le}\Longrightarrow \E_\Ex,
$$
which will yield Theorem \ref{C:1.12}. Recall
that the RVD condition \eqref{e:rvd} is equivalent to the existence of
${l_\mu} >1$ and ${\wt c_\mu}>1$ so that
$$
 V(x,{l_\mu}  r) \geq  {\wt c_\mu} V(x,r) \quad \hbox{for all } x\in M \hbox{ and } r>0,
$$ which implies that
\be\label{eq:vdconq}
\mu\big(B(x,{l_\mu} r)\setminus B(x,r)\big)>0 \quad \hbox{for each } x\in M \hbox{ and } r>0  .
\ee
For any set $A\subset M$, define
its first hitting time $\sigma_A :=\inf\{t>0: X_t\in A\}.$
 Recall that for two relatively compact open sets $A$ and $ B$ of $M$ with $A\Subset B$,
 $\Ca (A, B)$ is the $0$-order capacity in the transient Dirichlet form $(\sE, \sF_B)$, which is
 the Dirichlet form in $L^2(B; \mu)$ of the subprocess $X^B$ of the symmetric Hunt process $X$ killed upon leaving
 $B$. Thus  there exists a unique smooth measure $\nu:=\nu_{A,B}$ with ${\rm supp} [\nu] \subset \overline{A}$ such that
 $G_B \nu$ is the $0$-order equilibrium potential of $A$ in $(\sE, \sF_B)$ and
$\Ca(A,B)=\nu(\overline A)$, see e.g.  \cite[Page 87--88 and (3.4.3)]{CF} or the $0$-order version of \cite[(2.2.13)]{FOT}.
Here $G_B \nu (x):= \int_B G_B(x, y) \,\nu (dy)$.
This measure $\nu$ is called the relative capacitary measure  of $A$ in $B$. In particular, according to \cite[Corollary 3.4.3]{CF}
or \cite[Theorem 4.3.3]{FOT},
$$
G_B \nu (x)= \bP^x(\sigma_A<\tau_B) \quad  \hbox{for $\sE$-q.e. } x\in M.
$$

\begin{lemma}\label{L:ug}Assume that Assumptions $\ref{A1}$,
$\ref{A1-b}$ and $\EHI$ hold. Then
 $\E_{\Ex, \le}$ holds; that is,
there is a constant $c_1>0$ such that for almost all $x\in M$ and any $r>0$,  \begin{equation}\label{e-ug1}\bE^x\tau_{B(x,r)}\le c_1 \Ex(x,r).\end{equation}\end{lemma}

\begin{proof}
For $x_0\in M$ and $r>0$, let $D=B(x_0,2r)$.
 Let $\delta \in (0,1)$ be the constant in $\EHI$.
For any $y\in B(x_0,r)\setminus B(x_0,r/2)$, let
$\nu$  be the relative capacitary measure for
$B(y,\delta r/3)$
with respect to $G_D(x,y)$.
Then ${\rm supp} ( \nu ) \subset \overline{ B(y,\delta r/3) }$, where $\overline{A}$ denotes the closure of the set $A$ in $M$.
Applying $\EHI$
for $G_D(x_0,\cdot)$ on $B(y,r/2)$,  we have
\begin{equation*}\begin{split}
1 & \ge \bP^{x_0}(\sigma_{B(y,\delta r/3)}<\tau_D)
=\int_{\overline{B(y,\delta r/3)}}
G_D(x_0,z)\,\nu(dz)  \asymp G_D(x_0,y)
\nu(\overline{B(y,\delta r/3)})\\
  & = G_D(x_0,y) \Ca(B(y,\delta r/3),B(x_0,2r)) \ge G_D(x_0,y) \Ca(B(y,\delta r/3),B(y,4r))\\
  & \asymp    G_D(x_0,y) \Ca(B(x_0,r), B(x_0,2r)),\end{split}
\end{equation*}
  where the second inequality follows from the facts that $B(x_0,2r)\subset B(y,4r)$ and $\Ca(A,B)$ is decreasing
  in $B$, and
in the last
step we used \eqref{eq:wondf1-1}. Thus, \begin{equation}\label{*}
   G_{B(x_0, 2r)} (x_0,y)
   \le \frac{c_1}{\Ca(B(x_0,r), B(x_0,2r))}
  \quad\mbox{for all } ~y\in B(x_0,r)\setminus B(x_0,r/2).
	\end{equation}
In particular,   \begin{equation}\label{**}\begin{split}
  \int_{B(x_0,r)}G_{B(x_0,2r)}(x_0,y)\,\mu(dy)\le &\int_{B(x_0,r/2)} G_{B(x_0,2r)}(x,y)\,\mu(dy)\\
  &+  \frac{c_2 V(x_0,r)}{\Ca(B(x_0,r), B(x_0,2r))}.\end{split}\end{equation}

On the other hand, by the strong Markov property of $X^{B(x_0, 2r)}$,  for every $y\in B(x_0,r )$,
 \begin{align*}
 G_{B(x_0,2r)}(x_0,y)& = G_{B(x_0,r)}(x_0,y)+\bE^{y} \left[ G_{B(x_0,2r)}(x_0, X_{\tau_{B(x_0, r)}});   \tau_{B(x_0, r)} < \tau_{B(x_0, 2r)}\right]\\
   &\le G_{B(x_0,r)}(x_0,y)+ \bE^{y} \left[ G_{B(x_0,4r)}(x_0, X_{\tau_{B(x_0, r)}});   \tau_{B(x_0, r)} < \tau_{B(x_0, 2r)}\right]\\
 &\le G_{B(x_0,r)}(x_0,y)+ \frac{c_1}{\Ca(B(x_0,2r), B(x_0,4r))}\\
 &\le G_{B(x_0,r)}(x_0,y)+ \frac{c_3}{\Ca(B(x_0,r), B(x_0,2r))},
 \end{align*}
 where  in the second and the third inequalities, we used
 \eqref{*} and  \eqref{eq:wondf1-1},
 respectively. Hence,
 $$
 \int_{B(x_0,r/2)} G_{B(x_0,2r)}(x_0,y)\,\mu(dy)\le \int_{B(x_0,r/2)} G_{B(x_0,r)}(x_0,y)\,\mu(dy)+ \frac{c_4V(x_0,r)}{\Ca(B(x_0,r), B(x_0,2r))}.
 $$
 This together  with \eqref{**} yields that
\begin{align*}
\int_{B(x_0,r)}G_{B(x_0,2r)}(x_0,y)\,\mu(dy) &\le \int_{B(x_0,r/2)} G_{B(x_0,r)}(x_0,y)\,\mu(dy)+ \frac{c_5V(x_0,r)}{\Ca(B(x_0,r), B(x_0,2r))}\\
&=\int_{B(x_0,r/2)} G_{B(x_0,r)}(x_0,y)\,\mu(dy)+c_5 \, \Ex(x_0,r).
\end{align*}

By iterating the above estimate and using the reverse doubling property of $\Ex(x_0,r)$
in Assumption $\ref{A1}$(iv), we obtain that
\begin{align*}
\bE^{x_0}\tau_{B(x_0,r)}\le &\int_{B(x_0,r)}G_{B(x_0,2r)}(x_0,y)\,\mu(dy)
\le c_5\sum_{k=0}^\infty \Ex(x_0,2^{-k}r)\\
\le& c_5 \Ex(x_0,r)\sum_{k=0}^\infty 2^{-k\beta_1}= c_6 \, \Ex(x_0,r),
\end{align*}
where the second inequality is due to
$\bE^{x_0} \left[ \int_0^{\tau_D}  {\bf1}_{\{x_0\}}(X_t) \,dt \right]=0$, which is a consequence of Assumption \ref{A1-b}.
The proof is complete.  \qed  \end{proof}

 \begin{lemma}\label{L:lg}
Assume that Assumptions $\ref{A1}$,
$\ref{A1-b}$,
$\EHI$ and $\IJ_{\Ex,\le}$ hold. Then $\E_{\Ex, \ge}$ holds; that is,
there exists a constant $c_1>0$ such that
almost all $x\in M$ and any $r>0$, $$\bE^x\tau_{B(x,r)}\ge c_1 \Ex(x,r).$$\end{lemma}

\begin{proof} First note that by \eqref{e-ug1}, $\tau_{B(x,r)}<\infty$ a.s. $\bP^x$.
Next  by \eqref{vext}, we can choose
$l\ge 3\vee l_\mu$ large enough, where $l_\mu$ is in \eqref{eq:vdconq},
such that there is a constant $\rho \in(0,1)$ so that
$$
\Ex(x, r)\le \rho \Ex(x, lr) \quad \hbox{for every } x\in M \hbox{ and } r>0.
$$
Let $D=B(x,l^{k+1}r)$
where $k\ge1$ will be
determined later.
Note that the set $B(x,l^kr)\setminus B(x,r)$
is non-empty due to \eqref{eq:vdconq}.
According  to the L\'evy system, for fixed $x \in M_0$,
  \begin{align*}
  \bP^x(\sigma_{B(x,l^kr)\setminus B(x,r)}<\tau_D)
  &\ge\bP^x \left(
 X_{\tau_{B(x,r)}}\in B(x,l^kr) \right)\\
  &=1-\bP^x \left(
  X_{\tau_{B(x,r)}}\notin B(x,l^kr) \right)\\
  &\ge1-\bE^x\left[ \int_0^{\tau_{B(x,r)}}J(X_s, B(x,l^kr)^c)\,ds\right]\\
  &\ge1-c_1\frac{\Ex(x,r)}{\Ex(x,l^kr)}\ge 1-c_2 \rho^k,
  \end{align*}
  where in the third inequality we have used $\IJ_{\Ex,\le}$ and \eqref{e-ug1}, and the last inequality follows from Assumption $\ref{A1}$(iv). In particular, taking $k\ge1$ large enough such that $c_2\rho^k\le 1/2$, we have
  \be\label{e-wge}
  \bP^x \left(\sigma_{B(x,l^kr)\setminus B(x,r)}<\tau_D \right)\ge 1/2.
  \ee
Let $\delta\in (0,1)$ be the constant in $\EHI$.
By $\VD$, there exists
$L_*=L_*(l,k)\in \bN$
(independent of $x$ and $r$) such that $ \overline{B(x, l^k r)}\setminus B(x,r)\subset \cup_{i=1}^{L_*}B(x_i,\delta r/2)$ for some $x_i\in B(x,l^kr)\setminus B(x,r)$, $i=1,\cdots, L_*$, see e.g.,\ \cite[Lemma 3.1]{KT}.
Let $\nu$ be the relative
capacitary measure for $B(x,l^kr)\setminus B(x,r)$
with respect to $G_D(x,y)$,
which is supported on
$\overline{B(x, l^k r)}\setminus B(x,r)$.
By applying $\EHI$ for $G_D(x,\cdot)$ on each $B(x_i,r/2)$, we have
 \begin{equation}\label{e-ug}\begin{split}\bP^x(\sigma_{B(x,l^kr)\setminus B(x,r)}<\tau_D)&=G_D\nu(x)=
\int_{{ \overline{B(x, l^k r)}\setminus B(x,r) }}G_D(x,z)\,\nu(dz)\\
 &\le \sum_{i=1}^{L_*}\int_{B(x_i,\delta r/2)}G_D(x,z)\,\nu(dz)\\
 &\asymp \sum_{i=1}^{L_*}\mbox{ess\,inf}_{y\in B(x_i,\delta r/2)}G_D(x,y)
 \nu(B(x_i,\delta r/2)) \\
 &\le  \Ca(B(x, l^k r),D)
\sum_{i=1}^{L_*}\mbox{ess\,inf}_{y\in B(x_i,\delta r/2)}G_D(x,y)\\
 &\le c_3\Ca(B(x,r),B(x,2r))\sum_{i=1}^{L_*}\mbox{ess\,inf}_{y\in B(x_i,\delta r/2)}G_D(x,y),
 \end{split}\end{equation}
where in
the last inequality (with $c_3=c_3(l,k)>0$) is due to \eqref{eq:wondf1-1}.
Inequalities \eqref{e-wge} and \eqref{e-ug} imply that
\be\label{3-11-17}
\sum_{i=1}^{L_*}\mbox{ess\,inf}_{y\in B(x_i,\delta r/2)}G_D(x,y)\ge \frac{c_4}{\Ca(B(x,r),B(x,2r))}.
\ee
Noting that $B(x_i,\delta r/2)\subset B(x_i,r/2)\subset B(x,(l^k+1/2)r)\setminus B(x,r/2)$ for $1\le i\le L_*$,
we have
\begin{align*}
\int_{B(x,(l^k+1/2)r)\setminus B(x,r/2)} G_{D}(x,z)\,\mu(dz)
&\ge L_*^{-1}\sum_{i=1}^{L_*}\int_{B(x_i,\delta r/2)} G_{D}(x,z)\,\mu(dz)\\
&\ge L_*^{-1}\sum_{i=1}^{L_*}\mbox{ess\,inf}_{y\in B(x_i,\delta r/2)}G_D(x,y)\mu(B(x_i,\delta r/2))\\
&\ge c_5 \Ex(x,r),
\end{align*}
where in the last inequality we used
$\VD$ and \eqref{3-11-17}. Hence taking ${r^*}=l^{k+1}r$
and using  Assumption $\ref{A1}$(iv), we have
\begin{align*}
\E^x[\tau_{B(x,{r^*})}]\ge \int_{B(x,{r^*})\setminus
B(x,{r^*}/(2l^{k+1}))} G_{B(x,{r^*})}(x,z)\,\mu(dz)\ge c_5 \Ex(x,{r^*}/l^{k+1})
\ge c_6 \Ex(x,{r^*}).
\end{align*}
The proof is complete.
\qed
  \end{proof}

\noindent{\bf Proof of Theorem \ref{C:1.12}.}\quad According to Lemmas \ref{L:ug} and \ref{L:lg}, we know that under Assumptions \ref{A1}
and $\ref{A1-b}$,
$\EHI$ and $\IJ_{\Ex,\le}$
imply $\E_{\Ex}$. Using Remark \ref{intelem} and \eqref{vext}, we find that $\J_{\Ex,\le}$ implies $\IJ_{\Ex,\le}$. Note that in the statement of Corollary \ref{C:cor-1}, we indeed have that if $\J_{\phi}$ holds (without  $\E_\phi$), then
$$\FK(\phi)+\PI(\phi)+\CSJ(\phi)\Longrightarrow \EHI(\phi),$$ which can be seen from Theorem \ref{T:ehi-1} (i) and (ii). Therefore, the desired assertion follows from Corollary \ref{C:cor-1}.\qed

\section{Example: symmetric jump processes
of variable orders}\label{S:6}

In this section,
we apply the main results of this paper to show that
$\WEHI^+(\phi)$ and $\EHI (\phi)$
hold  for a symmetric stable-like process on $\bR^d$ of variable
order with state-dependent scale function $\phi (x, r)$.
This example is a modification from \cite[Example 2.3]{BKK}.
Let $\alpha:\bR^d\to [\alpha_1,\alpha_2]\subset(0,2)$ be such that
\begin{equation}\label{e:con}
|\alpha(x)-\alpha(y)|\le \frac{c}{\log (2/|x-y|)}  \quad \hbox{ for } |x-y|<1
\end{equation}
holds with
some constant $c>0$. Suppose that
$$
 \frac{c_1}{|x-y|^{d+\alpha(x)\wedge \alpha(y)}}\le J(x,y)\le \frac{c_2}{|x-y|^{d+\alpha(x)\vee \alpha(y)}} \quad
\mbox{for } |x-y|\le 1
$$
and
$$
J(x,y)\asymp \frac{1}{|x-y|^{d+\alpha_1}}\quad \mbox{for }|x-y|> 1.
$$

Define
 \begin{equation}
 \phi(x,r)= \begin{cases}
  r^{\alpha(x)}, &  0<r\le 1;\\
  r^{\alpha_1},&  r> 1.
\end{cases}
\end{equation}
We claim that $\phi (x, r)$ has properties \eqref{pvd1} and \eqref{pvd}.
Note that for every $x\in \bR^d$ and $0<r\le R$,
\begin{equation*}\begin{split}
 \frac{\phi(x,R)}{\phi(x,r)}&= \begin{cases}
 \left( \frac R{r}\right)^{\alpha(x)}, &  0<r\le R\le  1; \smallskip  \\
 \frac{R^{\alpha_1}}{r^{\alpha(x)}}, &  0<r\le1 < R; \smallskip \\
 \left( \frac R{r}\right)^{\alpha_1}, &  1<r\le R.
\end{cases}\end{split}
\end{equation*}
Since $\alpha (x)\in [\alpha_1, \alpha_2]$, it is clear that
$$ \left( \frac R{r}\right)^{\alpha_1} \le \frac{\phi(x,R)}{\phi(x,r)}\le   \left( \frac R{r}\right)^{\alpha_2};
$$
and so  $\phi(x,r)$ satisfies condition \eqref{pvd1}.
Note that $\sup_{0<r\leq 1} {\log (1/r)}/{\log (2/r)} <\infty$.
Thus by assumption \eqref{e:con}, there is a constant $c_3\geq 1$ so that
 for every $0<r\leq 1$ and for any $x,y\in \bR^d$ with $|x-y|\le r$,
$\phi (x, r) =r^{\alpha (x)}
\leq c_3 r^{\alpha (y) }= c_3 \phi (y, r)$.
When $r>1$, $\phi (x, r)=r^{\alpha_1} = \phi (y, r)$ for every $x, y\in \bR^d$.
Hence $\phi (x, r)$ satisfies  \eqref{pvd}.

We next verify that $\J_\phi$ holds
for $J(x, y)$.   Indeed,
by \eqref{e:con} there is a constant $c_4\ge 1$ such that
for all $x,y\in \bR^d$ with $|x-y|\le 1$,
$$1\le \frac{|x-y|^{\alpha(x)\wedge \alpha(y)}}{|x-y|^{\alpha(x)}} \le  |x-y|^{-|\alpha(x)-\alpha(y)|}= \exp\Big(|\alpha(x)-\alpha(y)|\log (1/|x-y|)\Big)\le c_4.$$
Similarly, there is a constant $c_5>0$ such that for all $x,y\in \bR^d$ with $|x-y|\le 1$,
$$
c_5\le \frac{|x-y|^{\alpha(x)\vee\alpha(y)}}{|x-y|^{\alpha(x)}} \le 1.
$$
Therefore by the definition of $\phi(x,r)$, we have
\begin{equation}\label{e:6.3}
J(x,y)\asymp \frac{1}{|x-y|^d \phi(x,|x-y|)},\quad x,y\in \bR^d;
\end{equation}
that is, $\J_\phi$ holds.

\medskip

Define a Dirichlet form $(\sE, \sF)$ on $L^2(\bR^d; dx)$ as follows:
$$
\sE(f,f)=
\int_{\bR^d\times \bR^d \setminus {\rm diag}}
(f(y)-f(x))^2J(x,y)\,dx\,dy$$ and $\sF$ is the closure the class of Lipschitz functions
on $\bR^d$ with
 compact support with respect to the norm
 $\sE_1(f, f)^{1/2}:= \left(\sE(f,f)+\|f\|^2_{L^2(\bR^d;dx)} \right)^{1/2}$.
 By \cite[Theorem 1.3]{BBCK}, there exists $\sN\subset \bR^d$ having zero capacity with respect to the Dirichlet form
$(\sE,\sF)$, and there is a conservative and symmetric Hunt process $X:=(X_t, t\ge0, \bP^x)$ with state space $\bR^d\setminus \sN$. Note that, in the present setting $X$ is a symmetric jump process
of variable
order.
Furthermore, since
$$
J(x,y)\ge   \frac{c_6}{|x-y|^{d+\alpha_1}},\quad x,y\in \bR^d,
$$
comparing it with rotationally symmetric $\alpha_1$-stable process, we have the following Nash inequality:
$$
\|f\|_2^{2+{2\alpha_1}/{d}}\le c_7 \sE(f,f)\|f\|_1^{{2\alpha_1}/{d}},\quad f\in \sF.
$$
 Hence, for any $x,y\in \bR^d\setminus \sN$ and $t>0$, the heat kernel $p(t,x,y)$ of the process $X$ exists and satisfies that
\begin{equation}\label{eff}
p(t,x,y)\le c_8 t^{-{d}/{\alpha_1}}.
\end{equation} By \cite[Theorem 3.5]{BKK}, we know that under the present setting, $p(t,x,y)$ can be chosen to be jointly continuous
in
$(x,y)$ for every fixed $t>0$. Therefore, without loss of generality, we can assume the exceptional set $\sN\subset \bR^d$ above to be empty, i.e.,\ $\sN=\emptyset.$

\begin{proposition}\label{P:6.1}
 Let $X$ be the process defined above.
Then the following holds.
\begin{itemize}
\item[{\rm(i)}] For any $x\in \bR^d$  and $r>0$,
$$\bE^x[\tau_{B(x,r)}]\asymp \phi(x,r).$$

\item[{\rm(ii)}] For
any $x\in \bR^d$  and $r>0$,
$$\Ca(B(x,r/2), B(x,r))\asymp r^d/\phi(x,r).$$

\item[{\rm(iii)}]
Both $\WEHI^+ (\phi)$ and $\EHI (\phi)$ hold
for the process $X$.
\end{itemize} \end{proposition}

\begin{proof}
(i) It is clear that there is a constant $c_1>0$ so that
$$
\int_{\{|x-y|\ge1\}} J(x,y)\,dy \le c_1 \quad \hbox{for every } x\in \bR^d.
$$
 Thus, according to \cite[Example 2.3 and Theorem 2.1]{BKK}, for all $x\in \bR^d$ and $0<r\le 1,$
$$\bP^x(\tau_{B(x,r)}\le t)\le c_2 t r^{-\alpha(x)}.$$ Hence, for all $x\in \bR^d$ and $0<r\le 1$,
\begin{equation}\label{eff1}\bE^x[\tau_{B(x,r)}] \ge \frac{1}{2c_2}r^{\alpha(x)}\bP^x \left(\tau_{B(x,r)}\ge \frac{1}{2c_2}r^{\alpha(x)}\right)\ge \frac{1}{4c_2}r^{\alpha(x)}.\end{equation}

 As mentioned just before
 this proposition,
 the heat kernel $p(t,x,y)$ of the process $X$ exists such that \eqref{eff} holds for all $x,y\in \bR^d$ and $t>0$. On the other hand, it immediately follows from the definition of $J(x,y)$ that there is a constant $c_3>0$ such that
 $$
 \sup_x\int_{B(x,r)}|x-y|^2J(x,y)\,dy\le  c_3 r^{2-\alpha_1}
 $$ and
  $$
 \sup_x\int_{B(x,r)^c}J(x,y)\,dy\le  c_3 r^{-\alpha_1}
 $$
 hold for all $r>1$. Then, by using the Davies argument and the Meyer decomposition, we can obtain that for all $t>0$
and
$x,y\in \bR^d$ with $|x-y|>1$,
\be\label{eq:fenoow}
p(t,x,y)\le
\frac{c_4}
{t^{d/\alpha_1}}
\left(1+\frac{|x-y|}{t^{1/\alpha_1}}\right)^{-d-\alpha_1}.
\ee
In fact, this can be verified by following the proof of \cite[Theorem 1.4]{BGK} line by line (see \cite[Subsection 3.2, Page 153--155]{BGK} for more details) as well as choosing the truncated constant $K>1$ and replacing $\alpha$ and $\beta$ by $d$ and $\alpha_1$ respectively.
Note that, upper bound estimate \eqref{eq:fenoow} says
that
for $x,y\in \bR^d$ with $|x-y|>1$, $p(t,x,y)$
is bounded, up to a constant multiple, by  that of rotationally symmetric $\alpha_1$-stable  processes on $\bR^d$.
Furthermore, it is easy to see that this process is conservative (see e.g.\ \cite[Theorem 1.1]{MUW}). Therefore, we can arrive at that
for all $x\in \bR^d$ and $r>1$,
\begin{equation}\label{eff2}\bE^x[\tau_{B(x,r)}] \ge c_5 r^{\alpha_1}.\end{equation} Indeed, by \eqref{eq:fenoow} and the conservativeness of the process $X$, we know that for all $x\in \bR^d$, $t>0$ and $r>1\vee t^{1/\alpha_1}$,
$$\bP^x(|X_t-x|\ge r)=\int_{B(x,r)^c}p(t,x,y)\,dy\le \frac{c_4}
{t^{d/\alpha_1}}\int_{B(x,r)^c}
\left(1+\frac{|x-y|}{t^{1/\alpha_1}}\right)^{-d-\alpha_1}\,dy\le \frac{c_6t}{r^{\alpha_1}}.$$ This along with the strong Markov property of $X$ yields that for all $x\in \bR^d$, $t>0$ and $r\ge 2(1\vee (2t)^{1/\alpha_1})$,
\begin{align*}\bP^x(\tau_{B(x,r)}\le t)\le &\bP^x(\tau_{B(x,r)}\le t,|X_{2t}-x|\le r/2)+\bP^x(|X_{2t}-x|\ge r/2)\\
\le &\sup_{z\in B(x,r)^c}\sup_{s\le t}\bP^z(|X_{2t-s}-z|\ge r/2)+\bP^x(|X_{2t}-x|\ge r/2)\\
\le& \frac{c_7t}{r^{\alpha_1}}.\end{align*} Hence, there are constants $r_0\ge1$ and $c_8>0$ such that for all $x\in \bR^d$ and $r\ge r_0$,
$$\bP^x(\tau_{B(x,r)}\le c_8r^{\alpha_1})\le  1/2.
$$
 In particular, for all $x\in \bR^d$ and $r\ge r_0$,
$$\bE^x[\tau_{B(x,r)}] \ge  c_8r^{\alpha_1}\bP^x(\tau_{B(x,r)}\ge c_8r^{\alpha_1})\ge {c_8}r^{\alpha_1} /2 .
$$
Note that for all $x\in \bR^d$ and $1<r\le r_0$, we have by \eqref{eff1}
$$\bE^x[\tau_{B(x,r)}] \ge \bE^x[\tau_{B(x,1)}] \ge \frac{1}{4c_2}.$$ Combining both estimates above, we prove \eqref{eff2}.

Now we consider  upper bound for $\bE^x[\tau_{B(x,r)}]$. First, assume $r\leq 1$.
Note that the sum
$\sum_{s\le  t\wedge \tau_{B(x,r)}}{\bf 1}_{\{|X_s-X_{s-}|>2r\}}$ is 1 if there is a jump of size at least $2r$ by
$t\wedge\tau_{B(x,r)}$,  in which case the process exits $B(x,r)$
by time $t$.  It is $0$ if there is no such jump. So, for all $y\in B(x,r)$,
\begin{align*}
\bP^y(\tau_{B(x,r)} \leq t)
&\ge  \bE^y \Big[\sum_{s\le  t\wedge \tau_{B(x,r)}}{\bf 1}_{\{|X_s-X_{s-}|>2r\}}\Big]\\
&=\bE^y \int_0^{t\wedge\tau_{B(x,r)}} \int_{B(X_s, 2r)^c}J(X_s,z)\,dz\,ds\\
&\ge\frac{c_9 \bE^y (t\wedge\tau_{B(x,r)})}{\phi(x,r)}\ge \frac{  c_9t\bP^y(\tau_{B(x,r)}>t)}{\phi(x,r)},
\end{align*}
 where in the second inequality we have used the facts that
$$\phi(x,r)\asymp\phi(y,r),\quad |x-y|\le r$$ and
$$
\int_{B(x,2r)^c} J(x,z)\,dz
\ge c_{10}\Big(
\int_{B(x,2)\setminus B(x,2r)} \frac{1}{|x-z|^{d+\alpha(x)}}\,dz
+\int_{B(x,2)^c} \frac{1}{|x-z|^{d+\alpha_1}}\,dz\Big)
\asymp r^{-\alpha(x)}.$$
Therefore,
$$\bP^y(\tau_{B(x,r)}>t)\le 1- \frac{c_{9} t\bP^y(\tau_{B(x,r)}>t)}{\phi(x,r)}.$$ Taking
$t=c_{9}^{-1} \phi(x,r)$
so that $\frac{c_{9} t}{\phi(x,r)}= 1$,
we obtain that for all $y\in B(x,r)$, $$\bP^y(\tau_{B(x,r)}>t)\le \frac{1}{2}.$$
Using the strong Markov property at time $mt$ for $m=1,2,\ldots,$
$$\bP^x(\tau_{B(x,r)}>(m+1)t)\le \bE^x\left( \bP^{X_{mt}}(\tau_{B(x,r)}>t);\tau_{B(x,r)}>mt\right)\le \frac{1}{2}\bP^x(\tau_{B(x,r)}> mt).$$ By induction
$\bP^x(\tau_{B(x,r)}> mt)\le 2^{-m}$. With this choice of $t$, we have that for all $x\in \bR^d$ and $r\in(0,1]$,
$$
\bE^x[\tau_{B(x,r)}]\le c_{11}  r^{\alpha(x)}.
$$

It is easily seen from \eqref{eff} that for all $x,x_0\in \bR^d$ and $t>0$,
$$
\bP^x(X_t\in B(x_0,r))\le c_{12} r^d t^{- {d}/{\alpha_1}}.
$$
For $r\ge 1$, taking
$t=(2c_{12})^{\alpha_1/d} r^{\alpha_1}$ so that
$c_{12} r^d t^{- {d}/{\alpha_1}}=1/2$, we find that
$$
\bP^x(\tau_{B(x_0,r)}> t)\le \bP^x(X_t\in B(x_0,r))\le \frac{1}{2}.
$$
Using the strong Markov property of $X$ again, we arrive at that for all $x,x_0\in \bR^d$,
$\bP^x(\tau_{B(x_0,r)}>kt)\le 2^{-k}$ and so for all $x\in \bR^d$ and $r>1$,
$$
\bE^x[\tau_{B(x,r)}]\le c_{13}  r^{\alpha_1}.
$$

(ii) This follows immediately from  Proposition \ref{P:ec} and the assertion (i).

(iii)  $\EHR$ holds by  \cite[Theorem 3.1]{BKK}.
On the other hand, as we noted in \eqref{e:6.3}, $\J_\phi$ holds, while $\E_{\phi}$ is established in (i).
Thus according to Propositions
\ref{P:fk}--\ref{P:CSJ},
we have $\FK(\phi)$, $\PI(\phi)$  and $\CSJ(\phi)$ for this symmetric non-local Dirichlet form.
The desired conclusion now follows from Corollary \ref{C:cor-1}.
  \qed\end{proof}

\medskip

\noindent{\bf Acknowledgement}. We thank the referee for  helpful comments on the paper.

\vskip 0.3truein

\noindent {\bf Zhen-Qing Chen}

\smallskip \noindent
Department of Mathematics, University of Washington, Seattle,
WA 98195, USA

\noindent
E-mail: \texttt{zqchen@uw.edu}

\bigskip

\noindent {\bf Takashi Kumagai}

\smallskip \noindent
 Research Institute for Mathematical Sciences,
Kyoto University, Kyoto 606-8502, Japan

\noindent Email: \texttt{kumagai@kurims.kyoto-u.ac.jp}

\bigskip

\noindent {\bf Jian Wang}

\smallskip \noindent
College of Mathematics and Informatics \& Fujian Key Laboratory of Mathematical
Analysis and Applications (FJKLMAA), 350007, Fuzhou, P.R. China.

\noindent Email: \texttt{jianwang@fjnu.edu.cn}

 \end{document}